%% file: 1101_cohom_of_certain_orb.tex
\documentclass{amsart}
\usepackage{latexsym}
\usepackage{amsmath,amsthm,amssymb, amscd,color}
\usepackage{etex, verbatim}
\usepackage[all]{xy}
\usepackage{framed}

\usepackage{tikz}
\usepackage{tikz-cd}
\usetikzlibrary{3d, matrix,decorations.pathreplacing,calc}
\usetikzlibrary{arrows,matrix,positioning}
\usetikzlibrary{fit}
\usetikzlibrary{patterns}
\usepackage{arydshln}
\input xy \xyoption{all}
\usepackage{blkarray} %This is for the indices of columns of row in a matrix. 

\usepackage{mathrsfs} %For mathscr

\usepackage[utf8]{inputenc}
\usepackage[normalem]{ulem}

\theoremstyle{plain}
\newtheorem{theorem}{Theorem}[section]
\newtheorem{prop}[theorem]{Proposition}

\newtheorem{corollary}[theorem]{Corollary}
\numberwithin{equation}{section}

\theoremstyle{definition}
\newtheorem{definition}[theorem]{Definition}
\newtheorem{example}[theorem]{Example}

\newtheorem{proposition}[theorem]{Proposition}
\newtheorem{remark}[theorem]{Remark}

\newcommand{\C}{\mathbb{C}}
\newcommand{\CP}{\mathbb{C}P}

\newcommand{\R}{\mathbb{R}}
\newcommand{\Z}{\mathbb{Z}}

\newcommand{\im}{{\rm im}\hspace{1pt}}

\newcommand{\coker}{{\rm coker}\hspace{1pt}}

\makeatletter \def\blfootnote{\xdef\@thefnmark{}\@footnotetext}\makeatother

\numberwithin{equation}{section}

\def\quasitoric{toric }

\def \({\left(}
\def \){\right)}
\def \<{\langle}
\def \>{\rangle}
\def \bar{\overline}
\def \deg{\mathrm{deg}}

\def\nd{\noindent}

\def \bb{\mathbb}

\def \mc{\mathcal}

\def \CC{{\bb{C}}}
\def \CP{{\bb{CP}}}

\def \NN{{\bb{N}}}
\def \RR{{\bb{R}}}

\def \ZZ{{\bb{Z}}}

\def \begineq{\begin{equation}}
\def \endeq{\end{equation}}

\numberwithin{equation}{section}

\begin{document}

\title{On integral cohomology of certain orbifolds}
% MSC classification
% 53C15; 53D20

\author[A. Bahri]{Anthony Bahri}
\address{Department of Mathematics, University of Rider, NJ, United States of America}
\email{bahri@rider.edu}

\author[D. Notbohm]{Dietrich Notbohm}
\address{Department of Mathematics, Vrije Universiteit Amsterdam,
 Amsterdam, The Netherlands}
\email{d.r.a.w.notbohm@vu.nl}

\author[S. Sarkar]{Soumen Sarkar}
\address{Department of Mathematics, Indian Institute of Technology 
Madras, Chennai, India}
\email{soumensarkar20@gmail.com}

\author[J. Song]{Jongbaek Song}
\address{Department of Mathematical Sciences, KAIST, Daejeon, Republic of Korea}
\email{jongbaek.song@gmail.com}

%\thanks{The third author was partially supported by Basic Science 
%	Research Program through the National Research Foundation of Korea (NRF) 
%	funded by the Ministry of  Science, ICT \& Future Planning 
%	(No. 2016R1A2B4010823).}

\subjclass[2010]{57R18, 57R19, 14M15}

\date{\today}

\keywords{orbifold, CW-complex, group action, 
homology group, cohomology ring, Grassmannian}

\abstract 
The CW structure of certain spaces, such as effective orbifolds, can be too
complicated for computational purposes.  In this paper we use the concept 
of  $\mathbf{q}$-CW complex structure on an orbifold, to detect torsion
in its integral cohomology. The main result can be applied to well known
classes of orbifolds or algebraic varieties having orbifold singularities,
 such as toric orbifolds, simplicial toric varieties, torus orbifolds and weighted Grassmannians. 
\endabstract

\maketitle

\section{Introduction}
Recent years have seen the development of an abundance 
of specialized tools for the study of orbifolds.  
A variety of cohomology theories  have been crafted 
specifically to investigate particular classes of orbifolds 
arising in complex geometry, algebraic geometry, 
symplectic geometry and string topology. 
Among these are: the {\em de Rham 
cohomology of orbifolds\/} \cite[Chapter 2]{ALR}, 
{\em orbifold Dolbeault cohomology\/} \cite{Bai54, Bai56, Sat56}, 
{\em Bott--Chern orbifold  cohomology\/} \cite{Ang13}, 
{\em orbifold quantum cohomology\/} 
\cite[Section $2.2.5$]{CCLT}, {\em intersection homology\/} 
\cite{GM} and the  {\em Chen--Ruan orbifold cohomology ring\/} 
\cite{CR}. 

These theories do yield information about the ordinary 
singular cohomology of orbifolds but the coefficients 
tend to be real, complex or rational, as appropriate. 
In general, the  singular integral cohomology ring of a 
given orbifold remains sometimes intractable object, 
difficult to compute. 

Among orbifolds for which the integral cohomology ring 
does succumb to the traditional methods of algebraic 
topology, are those having cohomology which can be 
shown to be concentrated in even degree. The prime 
example is that of weighted projective spaces. Additively, 
their integral cohomology agrees with that for ordinary 
projective spaces but the ring structure is complicated by 
an abundance of divisibility arising from the weights, 
\cite{Ka, Amr}, \cite[Section $5.1$]{BSS}. Included also 
in this class of ``nice'' orbifolds  are: weighted Grassmannians 
\cite{CR-wgr, AM}, certain singular toric varieties, toric 
orbifolds\footnote{In the literature, these are also called 
{\em quasi-toric orbifolds.\/}}, and orbifolds 
which arise as complex vector bundles over spaces 
with even cohomology.

With this in mind, the determination of {\em verifiable\/} 
sufficient conditions on an orbifold, which ensure that the  
integral cohomology is concentrated in even degrees, 
becomes a natural question.  Motivated  by Kawasaki's 
computation for weighted projective spaces, the approach 
taken here begins by identifying classes of orbifolds which 
can be constructed by a sequence of canonical cofibrations, 
in a manner not unlike that for CW complexes. We employ 
ideas originated by Goresky in \cite{Gor} and developed in 
a toric context by Poddar and Sarkar in \cite{PS}. 
Confounding the process however, is the ineluctable need 
to replace ordinary cells with ``{\bf q}-cells'', the quotient of 
a disc by the action of a finite group. The task reduces then 
to keeping at bay the intrusion of torsion into odd 
cohomological degrees.

The basic structure of a {\bf q}-CW complex is reviewed and 
developed further in Section \ref{sec_q-CW_complex}. 
We note in Proposition \ref{prop_qCW_simp_complex} that 
a {\bf q}-CW complex is homotopic to a simplicial complex. 
In order to compute 
the integral cohomology of a {\bf q}-CW complex inductively, 
using a specific sequence of cofibrations, we introduce the 
notion of a {\em building sequence\/} $\{(Y_{i},0_{i})\}_{i = 1}^{\ell}$. 
Here, $Y_{i}$ is obtained from $Y_{i-1}$  by attaching a 
{\bf q}-cell of the form $e^{{k_{i}}}/G_{i}$, where $e^{{k_{i}}}$ 
is an open disc and $G_{i}$ is a finite group acting linearly.  
The image of the origin $0_{i}$ in the {\bf q}-cell $e^{{k_{i}}}/G_{i}$ 
plays an important role.  

The results in the following sections  are anchored in the 
observation that when a finite group $G$ acts linearly on a 
closed disc $\overline{D}^{n}$, the torsion which can arise in 
$H_{\ast}(S^{n-1}\!/G; \mathbb{Z})$  depends on $|G|$ only. 
In particular, we conclude that if $p$ is co-prime
to all the primes dividing $|G|$, then 
$H_{\ast}(S^{n-1}\!/G; \mathbb{Z})$ has no $p$-torsion.
%\skp{0.1}

Section \ref{sec_hom_q-cw_comp} contains the main theorem 
about {\bf q}-CW complexes.
%\skp{0.1}
\begin{theorem}\label{thm_even_cells_no_p-torsion}
Let $X$ be a {\bf q}-CW complex with no odd dimensional 
{\bf q}-cells and $p$ a prime number. If $\{(Y_{i},0_{i})\}_{i = 1}^{\ell}$ 
is a building sequence such that $\gcd \{p, |G_{i}|\} = 1$ for all
$i$ with $e^{2k_{i}} /G_{i} = Y_{i}\setminus Y_{i-1}$, then 
$H_{\ast}(X; \mathbb{Z})$ has no $p$-torsion and 
$H_{\text{odd}}(X; \mathbb{Z}_{p})$ is trivial.
\end{theorem}
%\skp{0.1}
\nd Successive applications of this theorem yield its important theorem.
%\skp{0.1}
\begin{theorem}\label{thm_even_cells_no_torsion}
Let $X$ be a {\bf q}--CW complex with no odd dimensional {\bf q}--cells. 
If for each prime $p$ there is a building sequence 
$\{(Y_{i},0_{i})\}_{i = 1}^{\ell}$ such that $\gcd \{p, |G_{i}|\} = 1$ for all
$i$ with $e^{2k_{i}} /G_{i} = Y_{i}\setminus Y_{i-1}$, 
then $H_{\ast}(X; \mathbb{Z})$ has no torsion and 
$H_{\text{odd}}(X; \mathbb{Z})$ is trivial.
\end{theorem}
%\skp{0.1}

 Applications and examples are given in Section \ref{sec_Application_and_examples}. 
The topological construction of a toric orbifold  $X(Q,\lambda)$, arising
from a simple polytope $Q$ and an $\mathcal{R}$-characteristic 
pair $(Q,\lambda)$, is reviewed in Subsection \ref{toric_orbifolds}. 
Then, we re-visit the notion of a 
{\em retraction sequence}  of triples, $\{(B_{k},E_{k},v_{k})\}_{k=1}^{\ell}$, 
for a simple polytope $Q$, which was introduced in \cite{BSS}. 
Each space $B_{k}$ is a polytopal complex obtained by removing 
successively, carefully chosen vertices, beginning with $Q$ itself. 
Each $E_{k}$ is a particular face in $B_{k}$ and $v_{k}$ is a 
special vertex in  $E_{k}$ called a {\em free\/} vertex. 
Then Proposition \ref{ret_build} provides the means by which we can use a 
{\bf q}-CW structure to study toric orbifolds. 
%\skp{0.1}
%\begin{theorem}\label{thm_ret_seq_induces_build_seq}
%Each retraction sequence $\{(B_{k},E_{k},v_{k})\}_{k=1}^{\ell}$ of a simple polytope 
%$Q$ induces a building sequence $\{(Y_{i},0_{i})\}_{i = 1}^{\ell}$  for  the 
%toric orbifold $X(Q,\lambda)$ constructed with {\bf q}-cells of even dimension.
%\end{theorem}
%\skp{0.05}
 So, a toric orbifold which satisfies the hypothesis of Theorem \ref{thm_even_cells_no_torsion}
has cohomology which is concentrated in even degrees. 
Moreover, from the point of view of retraction sequences, the orders
$|G_{i}|$ appearing in Theorem \ref{thm_even_cells_no_torsion}, 
are computable explicitly from the $\mathcal{R}$-characteristic data 
$(Q,\lambda)$, which may be denoted as $g_{E_i}(v_i)$. This result appears in
Theorem \ref{no-tor-on-toric}.
%\skp{0.1}
%\begin{theorem}\label{thm_toric_orb_torsion_free}
%Let $X(Q,\lambda)$ be a toric orbifold. If for each prime $p$, there is a 
%retraction sequence $\{(B_{k},E_{k},v_{k})\}_{k=1}^{\ell}$ such that
%$\gcd \{p, g_{E_{i}}(v_{i})\} = 1$, then {$H^{\ast}\big(X(Q,\lambda);\mathbb{Z}\big)$}
%is torsion free and concentrated in even degree.
%\end{theorem}
%\skp{0.1}
 An example is given next which proves that the gcd condition in Theorem
  \ref{no-tor-on-toric} is {\em weaker\/} than the hypothesis of
\cite[Theorem $1.1$]{BSS}, the main result in that paper about the
 integral cohomology of toric orbifolds.
%\skp{0.1}

Our attention turns in Subsection \ref{toric_varieties} to the study of certain projective toric varieties. Following a brief review
of the basic construction of a toric variety $X_{\Sigma}$ from a fan $\Sigma$, we note that a 
{\em simplicial\/} {\em polytopal \/} fan arises as the {\em normal\/} fan of a simple polytope $Q$. Moreover, the fan determines an $\mathcal{R}$-characteristic pair $(Q,\lambda)$. In this case, we have  
$$X(Q,\lambda) \cong X_{\Sigma}$$

\nd as orbifolds with torus action. The results of Section \ref{sec_hom_q-cw_comp} are applied then to this class of toric varieties.
When the results are combined with \cite[Theorem 5.4]{BSS}, we get a complete description of the integral
cohomology ring under conditions {\em weaker\/} than those of \cite[Theorem 1.2]{BSS}.
%\skp{0.1}
%\begin{corollary}\label{cor_cohom_ring_of_toric_orb}
%Let $X_{\Sigma}$ be a projective toric variety associated to a normal fan $\Sigma$ 
%of a simple polytope $Q$ with $m$ facets. If for each prime $p$ there is a retraction sequence 
%$\{(B_{k},E_{k},v_{k})\}_{k=1}^{\ell}$ of $Q$ such that $\gcd \{p, g_{E_{i}}(v_{i})\} = 1$, then the cohomology ring of toric variety $X_{\Sigma}$ corresponding to $\Sigma$ is
%$$H^{\ast}(X_{\Sigma};\mathbb{Z})\cong w\mathcal{SR}[\Sigma]/\mathcal{J} 
%\subseteq \mathbb{Z}[x_{1},...,x_{m}]\big/\mathcal{I}+\mathcal{J},$$
%where $\mathcal{I}$ is the Stanley--Reisner ideal of $Q$ and $\mathcal{J}$ is the ideal generated by linear
%relations
%$$\sum_{i=1}^m\langle \lambda_{i},\mathbf{e}_{j}\rangle{x_{i}} = 0, \quad j = 1,\ldots,n,$$
%
%\nd where $\deg x_{i}  = 2$, $\mathbf{e}_{j}$ denotes the $j$-th standard unit vector in $\mathbb{Z}^{n}$ 
%and $\lambda_{i}$ is defined by \eqref{eq_linear_relations}.
%\end{corollary}

In Subsection \ref{subsec_torus_orb}, retraction sequence 
techniques are shown, under certain conditions, to apply 
equally well to a generalization of a toric orbifold called a 
{\it torus orbifold\/}, introduced by Hattori and Masuda\/ \cite{HM}. 
Here, the simple polytope $Q$ is replaced by a manifold 
with corners $P$ which has additional properties. 
As observed by Masuda and Panov \cite[Section $4.2$]{MP}, 
they too can be constructed by an $\mathcal{R}$-characteristic 
pair $(P,\lambda)$. This allows us to obtain a result for torus 
orbifolds similar to Theorem \ref{no-tor-on-toric}.

A discussion of weighted Grassmannians follows in Subsection 
\ref{subsec_wGr}. Abe and Matsumura \cite[Section 2]{AM} 
show that the Schubert cell structure, which is indexed by Young 
diagrams, determines a {\bf q}-cell structure on weighted 
Grassmannians. This allows for an application of Theorem 
\ref{thm_even_cells_no_p-torsion} at the end of this subsection. 

The paper concludes with a discussion of {\bf q}-CW complexes 
which may have odd dimensional {\bf q}-cells. 
The main theorem gives a condition on a prime $p$ and  a 
building sequence $\{(Y_{i},0_{i})\}_{i = 1}^{\ell}$
for a {\bf q}-CW complex $X$ which ensures that 
$H_{\ast}(X;\mathbb{Z})$ has no $p$-torsion.
Also a condition of having $p$-torsion in $H_{\ast}(X;\mathbb{Z})$
is given.

{\bf Acknowledgements}: This work was supported in part by grants 210386 and 426160 from Simons Foundation. The fourth author was partially supported by Basic
Science Research Program through the National Research Foundation of Korea (NRF) funded by the Ministry of Science, ICT $\&$ Future Planning (No. 2016R1A2B4010823)

\section{$\mathbf{q}$-CW complexes}\label{sec_q-CW_complex}
In this section,  we introduce the notion of $\mathbf{q}$-CW complex 
structure in which the analogue of an open cell  is the quotient of an open disk by an action 
of a finite group. The construction mirrors the construction of ordinary
CW complexes given, for example, in Hatcher \cite{Hat}. We note that 
$\mathbf{q}$-CW complex structures were used to compute the rational 
homology of certain singular spaces having torus symmetries in \cite{PS}.

\begin{definition}  
 Let $G$ be a finite group acting linearly on a closed
 $n$-dimensional disc $\bar{D}^{n}$ centered at the origin. Such an action
preserves $\partial{\bar{D}^n} = S^{n-1}$. We call the quotient $ \bar{D}^{n}/G $
 an $n$-dimensional $\mathbf{q}$-{\it disc}. If $\overline{e}^n$ is $G$-equivariantly
  homeomorphic to $\overline{D}^n$, we call $e^n/G$ a $\mathbf{q}$-{\em cell}. By abuse of notation,
 we shall denote the boundary of $\bar{e}^n$ by $S^{n-1}$ without confusion.
\end{definition}

A $\mathbf{q}$-CW complex is constructed inductively as follows.
Start with a discrete set $ X_{0} $, where points are regarded as $0$-dimensional 
$\mathbf{q}$-cells. Inductively, form the $n$-dimensional $\mathbf{q}$-skeleton
$ X_{n} $ from $ X_{n-1} $ by attaching finitely many $n$-dimensional $\mathbf{q}$-cells
$ \{ e^{n}/{G_{\alpha}}\}$ via continuous maps $\{\phi_{\alpha}\}$ where
 $ \phi_{\alpha} : S^{n-1}/G_{\alpha} \to X_{n-1} $. This means that $ X_{n}$
 is the quotient space of the disjoint
union $ X_{n-1}\bigsqcup_{\alpha}  \{\bar{e}^n/G_{\alpha}\} $ of $ X_{n-1} $ with a
finite collection of $n$-dimensional $\mathbf{q}$-disks
$ \bar{e}^{n}/G_{\alpha} $ under the identification
$ x\sim \phi_{\alpha}(x)$  for $ x\in S^{n-1}/G_{\alpha} $.

 If $X = X_n$ for some finite $n$, we call $X$ a {\em finite} $\mathbf{q}$-{\em CW complex}
 of dimension $n$.  The topology of $ X $ is the quotient topology built inductively.
 We say $Y$ is a $\mathbf{q}$-CW subcomplex of $X$, if $Y$
 is a $\mathbf{q}$-CW complex and $Y$ is a subset of $X$.

\begin{definition}
\begin{enumerate}
\item Let $0_{i} \in X$ denote the image of origin in the $\mathbf{q}$-cell $e^{k_i}/G_{i}$.
We call $0_i$ a {\it special point} corresponding 
to the $i$-th $\mathbf{q}$-cell. We may also refer to $0_{i}$ as a special
point of $X$.

\item Let $Y$ be a $\mathbf{q}$-CW subcomplex of a $\mathbf{q}$-CW complex $X$.
Then a special point $0_{i} \in Y$ is called \emph{free} in $Y$ if there is a
neighborhood of $0_{i}$ in $Y$ homeomorphic to $e^k/G_{i}$
for some $k \in \Z_{\geq 0}$ and finite group $G_i$.
\end{enumerate}
\end{definition}

\begin{remark}
\begin{enumerate}
\item Every finite $\mathbf{q}$-CW complex has a free special point. A finite 
$\mathbf{q}$-CW complex can be obtained by attaching one $\mathbf{q}$-cell at each time.
Attaching $\mathbf{q}$-cells need not be in increasing dimension. 
 
\item If $0_{i}$ is a free special point of a $\mathbf{q}$-CW complex $Y$,
then $Y$ can be obtained from a $\mathbf{q}$-CW subcomplex $Y'$ by attaching a
$\mathbf{q}$-cell $\bar{e}^{k_{i}}/G_{i}$. In this case, the natural deformation
retract from  $(\bar{e}^{k_{i}} \setminus \{0_{i}\})/G_{i}$ to $\partial{\bar{e}^{k_{i}}}/G_{i}$
induces a deformation retract from $Y \setminus \{0_{i}\}$ to $Y'$.

\item Let $S^{k-1}_{\frac{1}{2}} = \{x \in e^k \cong D^k ~|~ |x| = \frac{1}{2}\}$.
Then $S^{k-1}_{\frac{1}{2}}/G_i \cong S^{k-1}/G_i$ and the previous
remark implies that the following is a cofibration:
\begin{equation}
\begin{tikzcd}
S^{k-1}/G_i \arrow{r}{\phi_i} & Y'  \arrow[hookrightarrow]{r} & Y
\end{tikzcd}
\end{equation}
 where $\phi_i$ is defined by the deformation retraction.

\item If $k_{i}=0, 1, 2$ then $\bar{e}^{k_{i}}/G_{i}$ is homeomorphic
to a closed disc of dimension $k_{i}$. So attaching 
$\bar{e}^{k_{i}}/G_{i}$ is nothing but attaching a disc of dimension
 $0, 1$ or $ 2$. Hence we may assume $G_{i}$ is trivial in these cases.

\item Following Goresky \cite{Gor} one may obtain a CW structure on an effective
orbifold. Often however, it is too complicated for computational purpose.
\end{enumerate}
\end{remark}

\begin{prop}\label{prop_qCW_simp_complex}
If $X$ is a $\mathbf{q}$-CW complex, then $X$ is homotopic to
a simplicial complex. 
\end{prop}
\begin{proof}
We prove it by induction. By definition of $\mathbf{q}$-CW complex, $X_0$ has a
simplicial complex structure. Suppose by induction, the space $X_{i-1}$ has a
simplicial complex structure. Assume that $X_i$ is obtained by attaching a
$\mathbf{q}$-cell $\bar{e}^k/G_i$ via the map $\phi_i : S^{k-1}/G_i \to X_{i-1}$.
Using Theorem 3.6 of Illman \cite{Il}, we get a simplicial complex structure
on $\bar{e}^k/G_i$. Then by the simplicial approximation theorem (\cite[Theorem2C.1]{Hat}),
there is a homotopy from $\phi_i$ to a simplicial map $\xi_i : S^{k-1}/G_i \to X_{i-1}$.
Then one can complete the proof by following the arguments in the proof of Theorem 2C.5
of \cite{Hat}.
\end{proof}

Now, we introduce the following definition of \emph{building sequence} which
enables us to compute the integral cohomology ring of certain spaces with
$\mathbf{q}$-CW complex structure and detect the existence of $p$-torsion.

\begin{definition}
Let $X$ be a finite $\mathbf{q}$-CW complex with $\ell$ many 
$\mathbf{q}$-cells $e^{k_i}/G_i$ for $i \in \{1, \ldots, \ell\}$
and $1 \leq k_i \leq \dim X$. Let $Y_i$ be $i$-th stage 
$\mathbf{q}$-CW subcomplex of $X$ and $0_i$ a free special point 
in $Y_i$. We say $\{(Y_i, 0_i)\}_{i =1}^{\ell}$ is a {\em building sequence} for $X$.
\end{definition}   

We recall that the finite group $G$ acts on $\bar{D}^n$ linearly.
Thus the boundary $S^{n-1}$ is preserved by $G$-action. The quotient space $S^{n-1}/G$ is called an \emph{orbifold lens space} in \cite{BSS}.
We call a finite group $K$ a \emph{$|G|$-torsion} if $K$ is trivial or the prime factors 
of the order $|K|$ is a subset of the prime factors of $|G|$. 
Now, the transfer homomorphism \cite[III.2]{Bor} leads us the following lemma. 

\begin{prop}\label{prop_homology_orb_lens_sp}
The homology of an orbifold lens space $S^{n-1}/G$ is 
$$H_j(S^{n-1}/G;\ZZ)= \begin{cases} \Z & \text{ if } j=0,\\
a~ |G|\text{-torsion} &\text{ if } 1\leq j \leq n-2, \\
a ~|G|\text{-torsion} \text{ or } \ZZ & \text{ if } j=n-1.\end{cases} $$
In particular, if $G$-action preserves the orientation of $S^{n-1}$, 
then $H_{n-1}(S^{n-1}/G;\ZZ)\cong \ZZ$.  
\end{prop}
\begin{proof}	
We see $H_0(S^{n-1}/G;\ZZ)\cong \ZZ$ trivially. 
For $j=1, \dots, n-2$, consider the following isomorphism 
\begin{equation}\label{eq_transfer_isom}
H^\ast(X/G;\mathbf{k}) \cong H^\ast(X;\mathbf{k})^G,
\end{equation}
where $X$ is a locally compact Hausdorff space, $G$ is a finite group 
and $\mathbf{k}$ is a field of characteristic zero or coprime to $|G|$, 
see \cite[III.2]{Bor}. 

The result follows now from the universal coefficient theorem. 
In particular, if $G$-action preserves the orientation, then 
$H_{n-1}(S^{n-1}/G;\ZZ)\cong \ZZ$ because $S^{n-1}/G$ is orientable. 
\end{proof}

The universal coefficient theorem leads to the following. 
\begin{corollary}\label{no_tor_in_lens}
If $p$ is coprime to the prime factors $p_1, \dots, p_r$ of $|G|$, 
then the group $H_{\ast}(S^{n-1}/G;\ZZ)$ has
no $p$ torsion and $H_{j}(S^{n-1}/G; \ZZ_p)$ is trivial if $j \neq 0, n-1$.
\end{corollary}

We end this section by discussing the degree of an attaching map.
Let $X$ be a $\mathbf{q}$-CW complex and $\phi : S^k/G \to X_j \subseteq X$
an attaching map where $X_j$ is the $j$-dimensional $\mathbf{q}$-skeleton.
Then $\phi$ induces homomorphisms 
$$(\phi_{\ast})_k : H_k(S^k/G; \ZZ) \to H_k(X_j; \ZZ).$$
Now, $H_k(X_j; \ZZ)$ is isomorphic to $\ZZ^s \oplus K$ for some non-negative 
integer $s$ and finite group $K$. By Proposition \ref{prop_homology_orb_lens_sp},
we have $H_k(S^k/G; \ZZ)$ is finite or $\ZZ$. 
The next definition of {\it degree} will play an important role in Section \ref{sec_cells_in_every_dim}.

\begin{definition}\label{def_degree}
Let $H_k(S^k/G; \ZZ) \cong \ZZ$ and write $(\phi_{\ast})_k(1)=(d_1, \ldots, d_s, x)$ for 
$1\in \ZZ$. 
We define the \emph{degree} of $\phi$ by 
\begin{equation}\label{eq_degree_of_attaching}
\deg \phi:= \gcd\{d_i \mid 1\leq i \leq s, ~ d_i\neq 0 \}.
\end{equation}
If all $d_1, \ldots, d_s$ are zero or $H_k(X_j;\ZZ)$ has no free part, then we define degree of $\phi$ to be 0.
\end{definition}

%============================================================================
\section{Integral homology of $\mathbf{q}$-CW complexes with cells in even dimensions}
%\section{Proof of Theorem \ref{thm_even_cells_no_p-torsion}}
\label{sec_hom_q-cw_comp}

This section is devoted to the proof of Theorem \ref{thm_even_cells_no_p-torsion}, 
which gives a sufficient condition for determining the torsion in 
the cohomology of certain families of finite $\mathbf{q}$-CW complexes. 
Notice that one can verify Theorem \ref{thm_even_cells_no_torsion} 
by applying Theorem \ref{thm_even_cells_no_p-torsion} for each prime $p$.
%
%
%\begin{theorem}\label{no_p-torsion}
%Let $X$ be a $\mathbf{q}$-CW complex with no odd dimensional $\mathbf{q}$-cells and
%$p$ a prime number. If $\{(Y_i, 0_i )\}_{i=1}^m$ is a building 
%sequence such that $\gcd\{p, |G_i|\}=1$ for all $i$ with 
% $e^{2k_i}/G_i =Y_i \setminus Y_{i-1}$, then $H_{\ast}(X; \Z)$
% has no $p$-torsion and $H_{odd}(X; \ZZ_p)$ is trivial. 
% \end{theorem}
%

\begin{proof}[Proof of Theorem \ref{thm_even_cells_no_p-torsion}]
We prove the claim by the induction on the free special points in the sequence 
$\{(Y_i, 0_i)\}_{i=1}^\ell$. First, notice that if $i=1$, $Y_i$ is a point and if 
$i=2$ then $Y_i$ is a $\mathbf{q}$-CW complex obtained by attaching an even dimensional
$\mathbf{q}$-cell to a point. That is, $Y_2$ is homeomorphic to a suspension
on $S^{2k_i-1}/G_i$. Then by Proposition \ref{prop_homology_orb_lens_sp}, the 
claim is true.  
	
Now we assume that $H_{\ast}(Y_{i - 1}; \Z)$ has no $p$-torsion
 and $H_{odd}(Y_{i -1}; \ZZ_p)$ is trivial for $i >1$. To complete the induction,
 we shall prove that the same holds for $Y_i$. Note that $Y_i$
 can be constructed from $Y_{i-1}$ by attaching a $\mathbf{q}$-cell
 $\overline{e}^{2k_i}/G_i$ to it. So we have the cofibration:
\begin{equation}
\begin{tikzcd}
S^{2k_i-1}/G_i \arrow{r} & Y_{i-1} \arrow[hookrightarrow]{r} & Y_i.
\end{tikzcd}
\end{equation}

This cofibration induces the following long exact sequence in homology, 
\begin{equation}\label{eq_les_of_pair}
\begin{tikzcd}[row sep=tiny]
\cdots \arrow{r} & H_{j+1}(Y_{i}) \arrow{r} & \widetilde{H}_{j}(S^{2k_i -1}/G_i) \arrow{r} & \quad \\
H_{j}(Y_{i -1})  \arrow{r} &  H_{j}(Y_i) \arrow{r} & \widetilde{H}_{j-1}(S^{2k_i - 1}/G_i) \arrow{r} & \cdots.
\end{tikzcd}
%\xymatrix@R=.5pc
%{\quad \cdots \quad  \ar[r]&H_{j+1}(Y_{i}) \ar[r]& 
%\widetilde{H}_{j}(S^{2k_i -1}/G_i) \ar[r]& \\
%H_{j}(Y_{i -1})  \ar[r]& H_{j}(Y_i) \ar[r] &
%\widetilde{H}_{j-1}(S^{2k_i - 1}/G_i) \ar[r] & \cdots.}
\end{equation}

Suppose that $j$ is odd. By the induction hypothesis, the group $H_j(Y_{i - 1}; \Z)$
 has no $p$-torsion and $H_j(Y_{i -1}; \ZZ_p)$ is trivial. On the other hand,
 $\widetilde{H}_{j-1}(S^{2k_i - 1}/G_i; \Z)$  has no $p$-torsion and 
 $H_{j-1}(S^{2k_i -1}/G_i; \ZZ_p)$ is trivial  by Corollary \ref{no_tor_in_lens}.
 Therefore, from \eqref{eq_les_of_pair} we have $H_j(Y_{i}; \Z)$ has no $p$-torsion
 and $H_j(Y_{i}; \ZZ_p)$ is trivial.

Next, we assume that $j$ is even. Then, in the exact sequence 
\eqref{eq_les_of_pair}, the groups $\widetilde{H}_{j-1}(S^{2k_i - 1}/G_i; \Z)$,
$H_j(Y_{i - 1}; \Z)$ and $\widetilde{H}_{j}(S^{2k_i - 1}/G_i; \Z)$
have no $p$-torsion. Therefore $H_j(Y_{i}; \Z)$ has no $p$-torsion.
This completes the induction. 
\end{proof}

%\begin{corollary}\label{no_torsion}
%Let $X$ be a $\mathbf{q}$-CW complex with no odd dimensional
%$\mathbf{q}$-cells. If for each prime $p$ there is a building sequence 
%$\{(Y_i, 0_i)\}_{i \in [m]}$ such that $\gcd\{p, |G_i|\}=1$
%for all $i$ with $e^{2k_i}/G_i=Y_i \setminus Y_{i-1}$,
%then $H_{\ast}(X; \Z)$ has no torsion and $H_{odd}(X; \Z)$ is trivial. 
%\end{corollary}
%\begin{proof}
% The proof follows from the a successive application of Theorem \ref{thm_even_cells_no_p-torsion}.
%\end{proof}
%============================================================================

\section{Applications and Examples}\label{sec_Application_and_examples}
The goal of this section is to illustrate the results of Section \ref{sec_hom_q-cw_comp}
with some broad applications which improve upon certain previous results.

\subsection{Toric orbifolds}\label{toric_orbifolds}
Toric orbifolds were introduced in \cite{DJ} and explicitly studied in 
\cite{PS}\footnote{After  \cite{DJ}, \emph{toric manifold} was renamed in 
\cite{BP} as \emph{quasitoric manifold}, and  
the authors of \cite{PS} used \emph{quasitoric orbifolds} instead of 
\emph{toric orbifolds} of \cite{DJ}. }.
%Toric orbifolds are also known as quasitoric orbifolds.
In \cite{BSS}, the authors introduced the concept of a retraction 
sequence\footnote{Retraction sequences have strong connection with shellability of 
a simplicial complex. The authors of \cite{BSS} are preparing a separated 
paper about this. } for a simple polytope and discussed the integral homology and
cohomology of toric orbifolds. In this subsection we show that a retraction
sequence determines a $\mathbf{q}$-CW structure on the toric orbifold and then
we compare the main theorem of \cite[Theorem 1.1]{BSS} and Theorem
\ref{thm_even_cells_no_torsion}. For convenience we adhere the notation of \cite{BSS}.  

We begin by summarizing the definition of a \emph{retraction sequence}
of a simple polytope.  
Given an $n$-dimensional simple polytope $Q$ with $\ell$ vertices, 
we construct a sequence of triples $\{(B_k, E_k, b_k)\}_{1 \leq k \leq \ell}$, inductively. 

The first term is defined by $B_1 = Q=E_1$ and $b_1$ is a vertex of $B_1$.
Next, given $(k-1)$-th term $(B_{k-1}, E_{k-1}, b_{k-1})$, the next term 
$(B_{k}, E_{k}, b_{k})$ is defined by setting
\begin{equation}\label{eq_B_k_in_retraction_seq}
B_{k}=\bigcup \{E \mid  E \text{ is a face of } B_{k-1} \text{ and } b_{k-1}\notin V(E)\}.
\end{equation}
If a vertex $b_k$ exists in $B_k$ satisfying that $b_k$ has a neighborhood 
homeomorphic to $\RR^{d}_{\geq 0}$ as manifold with corners 
for some $d \in \{1, \dots, \dim B_{k}\}$, then we let $E_k$ be the unique face 
containing $b_k$ with $\dim E_k=d$. We call $b_k$ a \emph{free vertex} of $B_k$.
%We may think $B_k$ as the deletion of $b_{k-1}$ from $B_{k-1}$. 
Such a sequence is called a {\it retraction sequence} for $Q$ if the sequence ends
up with $(B_\ell, E_\ell, b_\ell)$ such that $B_\ell=E_\ell=b_\ell$ for some 
vertex $b_\ell \in V(Q)$. 

We call a simple polytope $Q$ \emph{admissible} if 
there exists at least one free vertex in each polytopal subcomplex 
$B_k$ defined in \eqref{eq_B_k_in_retraction_seq}. 
Hence, given $k$-th term $B_k$ as defined in \eqref{eq_B_k_in_retraction_seq}, 
any choice of free vertex of $B_k$ defines a retraction sequence. 
Two different retraction sequences of a prism are described in 
Figure \ref{fig_ret_of_prism} below, which shows that the prism is admissible by 
its symmetry. 

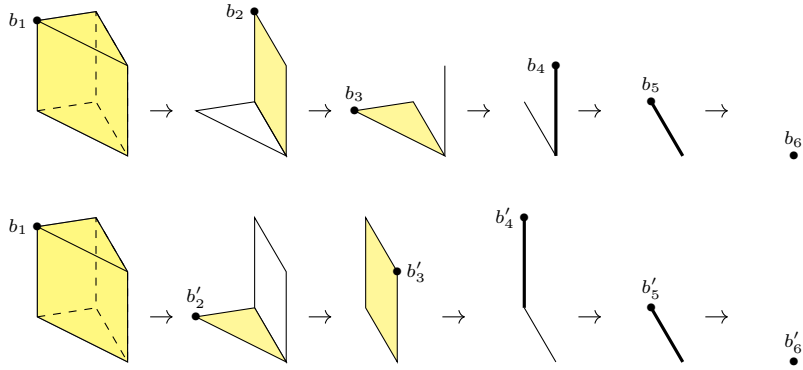
\begin{figure}[ht]
\begin{tikzpicture}[scale=0.6]
\draw[fill=yellow, opacity=0.5] (0,1)--(0,3)--(1.3,3.2)--(2,2)--(2,0)--cycle;
%\fill (0,3) circle [radius=2pt];
\node at (0,3) {\scriptsize$\bullet$};
\draw (0,1)--(2,0)--(2,2)--(0,3)--cycle;
\draw (0,3)--(1.3,3.2)--(2,2);
\draw[dashed] (1.3,3.2)--(2,2)--(2,0)--(1.3,1.2)--cycle;
\draw[dashed] (0,1)--(1.3,1.2);
\node[left] at (0,3) {\scriptsize $b_1$};
\draw[->] (2.5,1)--(3,1);

\begin{scope}[xshift=100]
\draw[fill=yellow!50] (1.3,3.2)--(2,2)--(2,0)--(1.3,1.2)--cycle;
\draw (0,1)--(1.3,1.2)--(2,0)--cycle;
\node at (1.3,3.2) {\scriptsize$\bullet$};
\node[left] at (1.3,3.2) {\scriptsize $b_2$};
\draw[->] (2.5,1)--(3,1);
\end{scope}

\begin{scope}[xshift=200]
\draw (2,2)--(2,0);
\draw[fill=yellow!50] (0,1)--(1.3,1.2)--(2,0)--cycle;
\node at (0,1) {\scriptsize$\bullet$};
\node[above] at (0,1) {\scriptsize $b_3$};
\draw[->] (2.5,1)--(3,1);
\end{scope}

\begin{scope}[xshift=270]
\draw[very thick] (2,2)--(2,0);
\draw (1.3,1.2)--(2,0);
\node at (2,2) {\scriptsize$\bullet$};
\node[left] at (2,2) {\scriptsize $b_4$};
\draw[->] (2.5,1)--(3,1);
\end{scope}

\begin{scope}[xshift=350]
\draw[very thick] (1.3,1.2)--(2,0);
\node at (1.3,1.2) {\scriptsize$\bullet$};
\node[above] at (1.3,1.2) {\scriptsize $b_5$};
\draw[->] (2.5,1)--(3,1);
\end{scope}

\begin{scope}[xshift=420]
\node at (2,0) {\scriptsize$\bullet$};
\node[above] at (2,0) {\scriptsize $b_6$};
\end{scope}

%%%%%%%%%Second ret. seq%%%%%%%%%%%%%%%

\begin{scope}[yshift=-130]
\draw[fill=yellow, opacity=0.5] (0,1)--(0,3)--(1.3,3.2)--(2,2)--(2,0)--cycle;
%\fill (0,3) circle [radius=2pt];
\node at (0,3) {\scriptsize$\bullet$};
\draw (0,1)--(2,0)--(2,2)--(0,3)--cycle;
\draw (0,3)--(1.3,3.2)--(2,2);
\draw[dashed] (1.3,3.2)--(2,2)--(2,0)--(1.3,1.2)--cycle;
\draw[dashed] (0,1)--(1.3,1.2);
\node[left] at (0,3) {\scriptsize $b_1$};
\draw[->] (2.5,1)--(3,1);

\begin{scope}[xshift=100]
\draw (1.3,3.2)--(2,2)--(2,0)--(1.3,1.2)--cycle;
\draw[fill=yellow!50] (0,1)--(1.3,1.2)--(2,0)--cycle;
\node at (0,1) {\scriptsize$\bullet$};
\node[above] at (0,1) {\scriptsize $b_2'$};
\draw[->] (2.5,1)--(3,1);
\end{scope}

\begin{scope}[xshift=170]
\draw[fill=yellow!50] (1.3,3.2)--(2,2)--(2,0)--(1.3,1.2)--cycle;
\node at (2,2) {\scriptsize$\bullet$};
\node[right] at (2,2) {\scriptsize $b_3'$};
\draw[->] (3,1)--(3.5,1);
\end{scope}

\begin{scope}[xshift=270]
\draw [very thick] (1.3,3.2)--(1.3,1.2);
\draw (1.3,1.2)--(2,0);
\node at (1.3,3.2) {\scriptsize$\bullet$};
\node[left] at (1.3,3.2) {\scriptsize $b_4'$};
\draw[->] (2.5,1)--(3,1);
\end{scope}

\begin{scope}[xshift=350]
\draw[very thick] (1.3,1.2)--(2,0);
\node at (1.3,1.2) {\scriptsize$\bullet$};
\node[above] at (1.3,1.2) {\scriptsize $b_5'$};
\draw[->] (2.5,1)--(3,1);
\end{scope}

\begin{scope}[xshift=420]
\node at (2,0) {\scriptsize$\bullet$};
\node[above] at (2,0) {\scriptsize $b_6'$};
\end{scope}

\end{scope}

\end{tikzpicture}
\caption{Two retraction sequences.}
 \label{fig_ret_of_prism}
 \end{figure} 

%$b_{k}$ is a free vertex
%of $B_{k}$ then $B_{k+1}$, the deletion of $b_{k}$ from $B_k$, has at least
%one free vertex. So far the authors do not know a simple polytope which is 
%not admissible. An example of two retraction sequences for a prism is given in Figure
%\ref{fig_ret_of_prism}. Note that a prism is an admissible simple polytope. 
%\begin{figure}[ht]
%        \centerline{
%           \scalebox{.55}{
%            \input{prismret.pdf_t}
%            }
%          }
% \caption{Two retraction sequences.}
% \label{fig_ret_of_prism}
% \end{figure} 

Next, we recall briefly the construction of \emph{toric orbifolds}. 
 A \emph{toric orbifold} is 
constructed from a combinatorial pair $(Q, \lambda)$, called an 
\emph{$\mathcal{R}$-characteristic pair}, of an $n$-dimensional simple 
convex polytope $Q$ and a function 
$$\lambda \colon \mathcal{F}(Q)=\{F_1, \dots , F_m\} \to \ZZ^n,$$
where $\mathcal{F}(Q)$ is the set of codimension one faces of $Q$, called 
\emph{facets}, and $\lambda$ satisfies the following condition:
\begin{equation}\label{eq_R-char_function_condition}
\{\lambda(F_{i_1}), \dots, \lambda(F_{i_k})\} 
\text{ is linearly independent, whenever } \bigcap_{j=1}^k F_{i_j}\neq \emptyset. 
\end{equation}
We call such a function $\lambda$ an \emph{$\mathcal{R}$-characteristic function} 
on $Q$. For the notational convenience, we sometimes denote 
$\lambda_i=\lambda(F_i)$ and call it an 
\emph{$\mathcal{R}$-characteristic vector} assigned to the facet $F_i$. 
 
\begin{remark}\label{rmk_simple_lattice_polytope}
One example of such a function can be provided by a 
\emph{simple lattice polytope} which is a simple polytope obtained by 
the convex hull of finitely many lattice points in $\Z^n\subset \R^n$. 
One can naturally assign each facet of simple lattice polytope 
its primitive normal vector as an $\mathcal{R}$-characteristic vector. 
Such a polytope is one of the key combinatorial objects in toric geometry,
as we shall explain briefly in the next subsection.
\end{remark}

We may regard $\ZZ^n$, the target space of $\lambda$, as the 
$\ZZ$-submodule of the Lie algebra of $T^n$. 
Let $E(x)  = F_{i_1} \cap \cdots \cap F_{i_k}$ denote the face with $x$ 
in its (relative) interior, and $T_{E(x)} \in T^n$ the subtorus 
generated by the images of $\lambda(F_{i_1}), \dots, \lambda(F_{i_k})$
under the map $ {\rm span}_\ZZ \{\lambda_{i_1},\dots, \lambda_{i_k}\} \hookrightarrow
 \ZZ^n \xrightarrow{\exp} T^n$.
Then, given an $\mc{R}$-characteristic pair $(Q,\lambda)$, we construct the following
quotient space
\begin{equation}\label{eq_def_of_toric_orb}
X(Q, \lambda):= Q\times T^n / \sim_\lambda,
\end{equation}
where $(x, t)\sim_\lambda (y, s)$ if and only if $x=y$ and $t^{-1}s \in T_{E(x)}$.

Now $X(Q, \lambda)$ has an orbifold structure 
induced by the group operation as described in \cite[Section 2]{PS}. The 
natural $T^n$-action, given by the multiplication on the second 
factor, induces  the orbit map 
\begin{equation}\label{eq_orbit_map}
\pi: X(Q, \lambda) \to Q  
\end{equation}
defined by  the projection onto the first factor. 
The space $X(Q, \lambda)$ is 
called the \emph{toric orbifold} associated to an $\mathcal{R}$-characteristic 
pair $(Q, \lambda)$. 
We remark that the authors of \cite{PS} gave an axiomatic definition 
 of \quasitoric orbifolds, which generalizes the definition 
 of toric manifolds of \cite[Section 1]{DJ}. 

Notice that the preimage $\pi^{-1}(F_{i})$ of a facet $F_i$ is a 
codimension 2 subspace fixed by a circle subgroup of $T^n$ 
generated by $\lambda_{i}$. Moreover, the 
preimage $\pi^{-1}(E)$ of a face $E=F_{i_1}\cap  \dots \cap  F_{i_k}$ 
is the intersection of $\pi^{-1}(F_{i_1}), \dots, \pi^{-1}(F_{i_k})$, 
which is again $T^n$-invariant. Indeed, it is shown in \cite[Section 2.3]{PS}
that $\pi^{-1}(E)$ is also a toric orbifold for each face $E$ of $Q$, 
whose $\mathcal{R}$-characteristic pair is described below.

As above, let $E=F_{i_1}\cap  \dots \cap  F_{i_k}$ be a face of codimension $k$ in 
$Q$, where $F_{i_1}, \dots, F_{i_k}$ are facets.  One can define a 
natural projection 
\begin{equation}\label{eq_def_of_rho_E}
\rho_E \colon \Z^n \to \Z^n / (({\rm span}_{\Z}\{\lambda_{i_1}, \dots, 
\lambda_{i_k}\}\otimes_\Z \R) \cap \Z^n),
\end{equation}
where $({\rm span}_\ZZ\{\lambda_{i_1}, \dots, \lambda_{i_k}\}\otimes_\Z \R) 
\cap \Z^n$ is a free $\Z$-module of rank $k$. 
We regard $E$ as an independent simple 
polytope whose facets $\mathcal{F}(E)$ are of the form:
$$\mathcal{F}(E) = \{ E\cap F_j \mid F_j \in \mathcal{F}(Q)~ 
\text{and}~ j \notin \{ i_1, \dots, i_k\}, ~\text{and}~ E\cap F_j\neq 
\varnothing\}.$$

Now, the composition of the maps $\rho_E$ and $\lambda$ yields 
an $\mathcal{R}$-characteristic function for $E$ 
\begin{equation}\label{eq_def_of_induced_char_ftn}
\lambda_E \colon \mathcal{F}(E) \to \Z^{n-k},
\end{equation}
defined for $\lambda_E(E\cap F_j)$ to be 
the primitive vector of $(\rho_E\circ \lambda)(F_j)$.  
Indeed, condition \eqref{eq_R-char_function_condition}  for the 
$\mathcal{R}$-characteristic function $\lambda_E$ follows naturally from $\lambda$. 

Next, we 
define certain integers associated to each vertex.  
Let $v$ be a vertex of $Q$ and $E$ a face of codimension $k$ 
containing $v$. Then, 
there are facets $F_{j_1}, \ldots, F_{j_{n-k}}$
such that $v= (E \cap F_{j_1}) \cap \cdots \cap (E \cap F_{j_{n-k}}).$ 
We define
\begin{equation}\label{eq_g_E(v)}
g_E(v) := \big| \det\left[ 
\lambda_E(E \cap F_{j_1})^t, \ldots, \lambda_E(E \cap F_{j_{n-k}})^t
\right] \big|,
\end{equation}
where $t$ denotes the transpose. 
In particular, when $E=Q$ and $v=F_{i_1}\cap \cdots \cap F_{i_n}$, 
the number $g_Q(v)$ is
\begin{equation}\label{eq_g_Q(v)}
g_Q(v)=\big| \det  \left[ \lambda(F_{i_1})^t \cdots \lambda(F_{i_n})^t \right] \big|.
\end{equation}
Given an $\mathcal{R}$-characteristic pair $(Q, \lambda)$ and 
each term $(B_i, E_i, b_i)$ of a retraction sequence of $Q$,
we denote naturally $g_{B_i}(v):=g_{E_i}(v)$ for each free vertex $v$ of $B_i$,
which makes the statement of Theorem \ref{thm_gcd_con} simpler. 

\begin{remark}
The number $g_E(v)$ encodes the singularity of $X(E, \lambda_E)$. 
Indeed, the local group
of the orbifold chart around the point $\pi_E^{-1}(v)$, where 
$\pi_E \colon X(E, \lambda_E) \to E$ is the orbit map, has the order 
$g_E(v)$. 
\end{remark}

The next theorem gives a sufficient condition ensuring that the cohomology 
of a toric orbifold is concentrated in even degree and torsion free. 

\begin{theorem}{\cite[Theorem 1.1]{BSS}}\label{thm_gcd_con}
Given any toric orbifold $X(Q,\lambda)$ over an admissible simple polytope, assume that
$$\gcd \big\{ g_{B_i}(v) \mid v ~\text{is a free vertex in}~B_i \big\}=1,$$
 for each $B_i$ which appears in a retraction sequence of $Q$ with $\dim B_i > 1$;
 then $H^\ast(X(Q, \lambda) ;\ZZ)$ is torsion free and concentrated in even degrees. 
\end{theorem}

We remark that \cite[Theorem 1.1]{BSS} was proved for toric orbifolds over
admissible simple polytopes, but it was not mentioned explicitly.
Retraction sequences and building sequences are related by the following. 

\begin{prop}\label{ret_build}
If $\{(B_i, E_i, b_i)\}_{i=1}^\ell$ is a retraction sequence of an $n$-dimensional simple 
polytope $Q$ and $(Q, \lambda)$ is an $\mathcal{R}$-characteristic pair, then
it induces a building sequence $\{(Y_{i}, 0_{i})\}_{i=1}^{\ell}$
of $X(Q, \lambda)$. 
\end{prop}

\begin{proof}
Let $E_i$ be a face of $Q$ with $\dim E_i=n-k$,  and $U_i$ the open subset of 
$E_i$ obtained by deleting the faces of $E_i$ not containing the vertex $b_i$.
Observe that $b_i=\bigcap_{j=1}^{n-k} (E_i \cap F_{i_j})$ for some facets $F_{i_1}, \ldots,
F_{i_{n-k}}$ of $Q$ and $U_i$ is homeomorphic to $\RR^{n-k}_{\geq 0}$ as manifold
with corners. Note that $\pi^{-1}(E_i)$ is a toric orbifold and homeomorphic to
$X(E_i, \lambda_{E_i})$, see \cite[Proposition 3.3]{BSS}. So, we can apply the 
orbifold chart construction  on a toric orbifold described in \cite[Subsection 2.1]{PS} 
to get that $\pi^{-1}(U_i)$ is homeomorphic to the quotient of an open disc
$D^{2(n-k)}$ by a finite group 
$G_i$ which is given by  $$G_{i} = \ZZ^{n-k} / \text{span}_\ZZ
\{ \lambda_{E_i}(E_i\cap F_{i_1}), \ldots, \lambda_{E_i}(E_i\cap F_{i_{n-k}}) \}.$$
Since $Q = \bigcup_{i=1}^{\ell} U_i$, we get a $\mathbf{q}$-CW complex structure
on $X(Q, \lambda)$ with $\mathbf{q}$-cells $\pi^{-1}(U_i)$ for $i=1, \ldots, \ell$.
Setting $Y_i = \bigcup_{j \geq \ell -i+1} (\pi^{-1}(U_j))$ gives a building sequence
 $\{(Y_i, 0_i)\}_{i=1}^{\ell}$ where $0_i = \pi^{-1}(b_i)$ for $i=1, \ldots, \ell$. 
\end{proof}
We note that the order $|G_i |$ is exactly $g_{E_i}(b_i)$
defined as in \eqref{eq_g_E(v)}.

\begin{prop}\label{prop_new_theorem_is_better}
If a toric orbifold $X(Q, \lambda)$ satisfies the assumption of Theorem
 \ref{thm_gcd_con}, then it satisfies the assumption of Theorem \ref{thm_even_cells_no_torsion}.
\end{prop}
\begin{proof}
Let $X(Q, \lambda)$ be a toric orbifold over an $n$-dimensional admissible
simple polytope $Q$ with $\ell$ vertices, and $p$ a prime. Since $B_1=Q$, we have 
$$\gcd\{g_Q(v) \mid  v \in V(Q)\}=1$$ 
and so there is a $b_1 \in V(Q) $ such that $\gcd\{p, g_Q(b_1)\}=1$.
We consider $(B_1, E_1, b_1) = (Q, Q, b_1)$ for the first step of the retraction sequence.
Since $Q$ is simple, $B_2$ has $n$ many free vertices, say $\{v_{i_1}, \ldots, v_{i_n}\}$,
namely, the ones at the \emph{ends} of the $n$ deleted edges joining $b_1$.
The assumption of Theorem \ref{thm_gcd_con} implies that  
$$\gcd\{g_{B_2}(v_{i_1}), \ldots, g_{B_2}(v_{i_n}) \} = 1,$$ 
which guarantees that there is a vertex $b_2 \in \{v_{i_1}, \ldots, v_{i_n}\}$
with $\gcd\{p, g_{B_2}(b_2)\}=1$. Let $E_2$ be the face of $B_1=Q$ 
determined by the edges of $B_2$ which intersect $b_2$. 
So $(B_2, E_2, b_2)$ can be a second term of a retraction sequence.
Since $Q$ is an admissible simple polytope, then $B_3$ has at least one free vertex. 
Continuing this process, since the number of vertices $\ell=|V(Q)|$ is finite, 
one get a retraction $\{(B_i, E_i, b_i)\}_{i=1}^\ell$
of $Q$ such that $\gcd\{p, g_{E_i}(b_i)\}=1$ for $i=1, \ldots, \ell$. Therefore
the result follows from Proposition \ref{ret_build} and its proof. 
\end{proof}

The next theorem is the direct application of Theorem \ref{thm_even_cells_no_torsion} to 
toric orbifolds. 

\begin{theorem}\label{no-tor-on-toric}
Let $X(Q, \lambda)$ be a toric orbifold. If, for each prime $p$, there is a
retraction sequence $\{(B_i, E_i, v_i)\}_{i=1}^\ell$ of $Q$ such that $\gcd\{p, g_{E_i}(v_i)\}=1$,
then $H^\ast(X(Q, \lambda);\ZZ) $ is torsion free and concentrated in even degrees. 
\end{theorem}
\begin{proof}
The proof follows from Proposition \ref{ret_build} and Theorem \ref{thm_even_cells_no_torsion}.
\end{proof}

The next example illustrates the fact that the converse of 
Proposition \ref{prop_new_theorem_is_better} is not true in general. 
Hence, Theorem \ref{no-tor-on-toric} generalizes Theorem \ref{thm_gcd_con}.

\begin{example}\label{ex_converse_not_true}
In this example we construct a toric orbifold which satisfies the condition
of Theorem \ref{no-tor-on-toric}, but fails to satisfy the hypothesis of
Theorem \ref{thm_gcd_con}. Consider the $\mathcal{R}$-characteristic pair
$(Q, \lambda)$ illustrated in Figure \ref{fig-eg1}. 

\begin{figure}[ht]
\begin{tikzpicture}
\draw[thick] (0,0)--(3,0)--(3,3)--(0,3)--cycle;
\draw[thick] (4,4)--(4,1)--(3,0);
\draw[thick, dashed] (1,1)--(0,0); \draw[thick, dashed] (1,1)--(1,4); \draw[thick, dashed](1,1)--(4,1);
\node[left] at (0,3) {\footnotesize $v_{145}$};
\node[below right] at (3,3) {\footnotesize $v_{125}$};
\node[right] at (4,4) {\footnotesize $v_{235}$};
\node[above] at (1,4) {\footnotesize $v_{345}$};
\node[left] at (0,0) {\footnotesize $v_{146}$};
\node[right] at (3,0) {\footnotesize $v_{126}$};
\node[right] at (4,1) {\footnotesize $v_{236}$};
\node[above left] at (1,0.9) {\footnotesize $v_{346}$};

\node[fill=white] at (1.5,1.8) {\footnotesize $\lambda(F_1)=(0,1,2)$};
\node[fill=white] at (2,3.5) {\footnotesize $\lambda(F_5)=(1,0,0)$};
\draw[thick] (0,3)--(1,4)--(4,4)--(3,3);

\node[right] at (5,0) {\footnotesize $\lambda(F_6)=(2,1,0)$};
\draw [<-] (5,0) [out=210, in=300] to (2.2,0);
\draw[dotted] (2.2,0)--(2,0.5);

\node[right] at (5,1.5) {\footnotesize $\lambda(F_2)=(0,1,0)$};
\draw[<-] (5,1.5) -- (3.5,2);

\node[right] at (5,2.5) {\footnotesize $\lambda(F_3)=(0,0,1)$};
\draw [<-] (5,2.5) [out=140, in=0] to (4,3);
\draw[dotted] (4,3) [out=180, in=60] to (2.5,2.5);

\node[left] at (-1,2) {\footnotesize $\lambda(F_4)=(0,2,1)$};
\draw [<-] (-1,2) [out=40, in=180] to (0,2.5);
\draw [dotted] (0,2.5)--(0.5, 2.5);
\end{tikzpicture}
\caption{An $\mathcal{R}$-characteristic function on 3-cube $Q$.}
 \label{fig-eg1}
\end{figure}
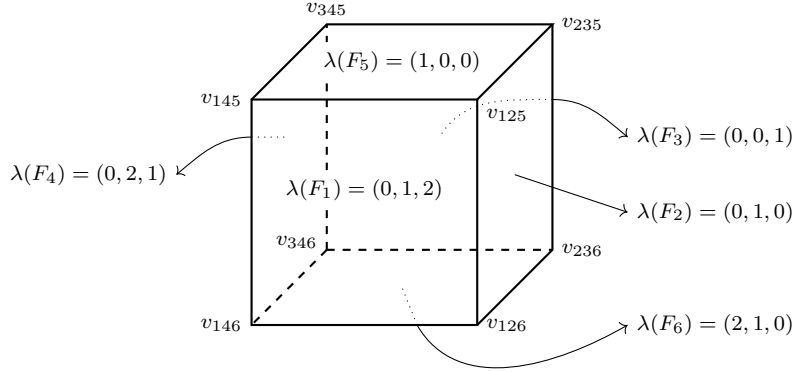

%
%\begin{figure}[ht]
%        \centerline{
%           \scalebox{.50}{
%            \input{torus.pdf_t}
%            }
%          }
% \caption{An $\mathcal{R}$-characteristic function on 3-cube $Q$.}
% \label{fig-eg1}
% \end{figure}  
One can compute that 
\begin{align*}
g_Q(v_{125})=2,~ g_Q(v_{235})=1,~ g_Q(v_{345})=2,~g_Q(v_{145})=3,\\
g_Q(v_{126})=4,~ g_Q(v_{236})=2, ~g_Q(v_{346})=4,~g_Q(v_{146})=6.
\end{align*}
Now we consider two different retraction sequences of $Q$, 
as in  Figure \ref{fig_ret_of_cube}. The first retraction sequence is given by
\begin{align}\label{eq_ret_(a)}
\begin{split}
&(B_1, B_1, v_{235}) \to (B_2, F_6, v_{236})\to (B_3, F_4, v_{345}) \to (B_4, F_1, v_{125})\to \\
&(B_5, F_4\cap F_6, v_{346})\to (B_6, F_1 \cap F_4, v_{145})\to 
(B_7, F_1 \cap F_6, v_{126}) \to\\
& (B_8, v_{146}, v_{146}),
\end{split}
\end{align}
which is illustrated in Figure \ref{fig_ret_of_cube}-Retraction (a). 
In the second term of the sequence \eqref{eq_ret_(a)}, we compute
 $g_{F_6}(v_{236})$ as follows; since 
$$(\text{span}\{\lambda(F_6)\}\otimes_\ZZ \RR)\cap \ZZ^3 =\text{span}\{(2,1,0)\},$$
a choice of basis $\{(1,0,0), (2,1,0), (0,0,1)\}$ of $\ZZ^3$ 
induces the $\mathcal{R}$-characteristic function 
$$\lambda_{F_6} \colon \mathcal{F}(F_6)=\{F_i\cap F_6\mid i=1, 2,3,4\} 
\to \ZZ^3/\text{span}\{(2,1,0)\}\cong \ZZ^2,$$
on $F_6$  defined as in \eqref{eq_def_of_induced_char_ftn}  by 
$\lambda_{F_6}(F_1\cap F_6)=(-1,1)$, 
$\lambda_{F_6}(F_2\cap F_6)=(-1,0)$, 
$\lambda_{F_6}(F_3\cap F_6)=(0,1)$ and 
$\lambda_{F_6}(F_4\cap F_6)=(-4,1).$
In particular, 
$$g_{F_6}(v_{236})=\left| \det \begin{bmatrix} \lambda_{F_6}(F_2\cap F_6)^t  &
\lambda_{F_6}(F_3\cap F_6)^t \end{bmatrix} \right|
=\left| \det \begin{bmatrix} -1 & 0 \\ 0 & 1\end{bmatrix} \right| = 1.$$
Similarly, one can compute the integers defined in \eqref{eq_g_E(v)} with respect to 
the retraction sequence \eqref{eq_ret_(a)} as follows:
$$g_{F_4}(v_{345})=g_{F_1} (v_{125})=g_{F_4\cap F_6}(v_{346})=
g_{F_1 \cap F_4}(v_{145})=g_{F_1 \cap F_6}(v_{126})=1.$$
Hence, the $\mathcal{R}$-characteristic pair $(Q, \lambda)$ defined in 
Figure \ref{fig-eg1} satisfies the hypothesis of Theorem \ref{no-tor-on-toric}. 
\begin{figure}
        \centerline{
           \scalebox{.60}{
            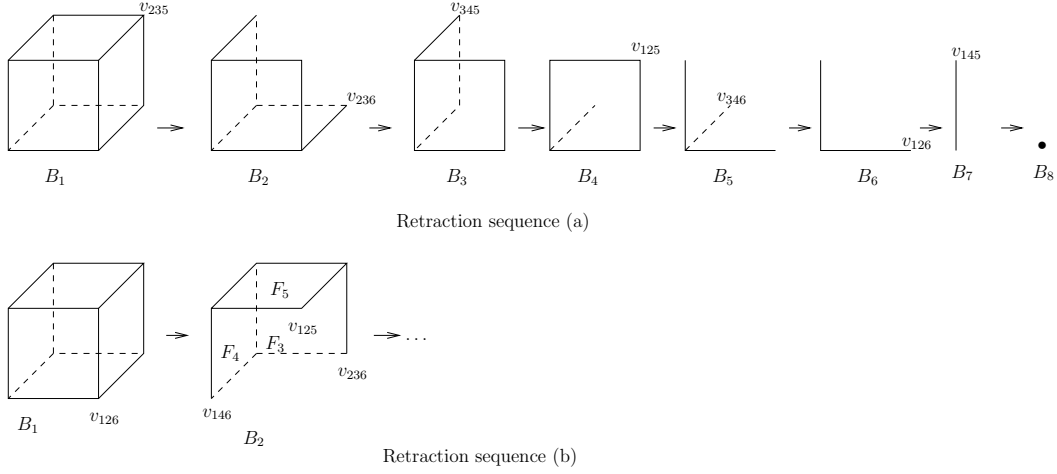
            }
          }
 \caption{Two retraction sequences of $Q$.}
 \label{fig_ret_of_cube}
 \end{figure} 
 
Next, we consider a retraction sequence whose first choice of free vertex
is $v_{126}$, see Figure \ref{fig_ret_of_cube}-Retraction (b). The vertices
$v_{125}, v_{146}, v_{236}$ are free vertices in $B_2$. Similarly
as above, one can compute that 
$g_{F_4}(v_{146})=2,~g_{F_5} (v_{125})=2,$ and $g_{F_3}(v_{236})=2.$
Hence,  
$$\gcd\{g_{F_4}(v_{146}), g_{F_5} (v_{125}), g_{F_3}(v_{236})\} =2,$$
which fails to satisfy the assumption of Theorem \ref{thm_gcd_con} for $B_2$. 
\end{example}

%=================================================================

%%%%%%%%%%%%%%%%%%%%%%%%%%%%%%%%%%%%%%%%%
%%%% Toric variety description from Lattice points
%%%%%%%%%%%%%%%%%%%%%%%%%%%%%%%%%%%%%%%%%
%\subsection{Toric varieties}\label{toric_varieties}
%Toric varieties are important objects in algebraic geometry and they have been
%studied from the beginning of nineteen-seventies. Among several different 
%approaches to construct a toric variety, we briefly introduce the 
%construction from the collection of lattice points in a vector space $\RR^n$. 
%
%Let $m=(a_1, \dots, a_n)\in \ZZ^n$ be a lattice point in $\RR^n$. We 
%begin by associating $m$ with a character
%$$\chi^m \colon (\CC^\ast)^n \to \CC^\ast$$
%defined for $\mathbf{t}=(t_1, \dots, t_n)\in (\CC^\ast)^n$ 
%by $\chi^m(\mathbf{t})=(t_1^{a_1}, \dots, t_n^{a_n})$. 
%Then, the finite set $\mathcal{A}=\{m_1, \dots, m_\ell\}\subset \ZZ^n$ 
%of lattice points in $\RR^n$ yields the following map: 
%$$\phi^\mathcal{A} \colon (\CC^\ast)^n \to \CP^{\ell}$$
%defined by $\phi^\mathcal{A} (\mathbf{t}) = 
%[\chi^{m_1}(\mathbf{t}); \dots; \chi^{m_\ell}(\mathbf{t})]$. 
%
%The projective \emph{toric variety} $X_\mathcal{A}$ associated to $\mathcal{A}$ 
%is defined by the Zariski closure of the image of the map 
%$\phi^{\mathcal{A}}$. 

%%%%%%%%%%%%%%%%%%%%%%%%%%%%%%%%%%%%%%%%%
%%%% Toric variety description from fan
%%%%%%%%%%%%%%%%%%%%%%%%%%%%%%%%%%%%%%%%%
\subsection{Toric varieties}\label{toric_varieties}
Toric varieties are important objects in algebraic geometry which have been
studied from the beginning of nineteen-seventies. 
They featured prominently in early investigations of the phenomenon of mirror symmetry. 
One approach to constructing them begins with
a lattice $N(\cong \ZZ^n)$ and a \emph{fan} $\Sigma$  which is a finite collection of 
\emph{cones} in a vector space $N_\RR:=N \otimes_\ZZ \RR (\cong \RR^n)$
satisfying certain conditions, see \cite[Chapter 1--3]{CLS}, \cite[Chapter 5]{Ew}, 
\cite[Chapter 1]{Ful-ITV} and  \cite[Chapter 1]{Oda} for details.  
 
We begin by reviewing briefly the definition of cones and fans.  
A \emph{cone} is a positive hull of finitely many elements in $N_\RR$.
To be more precise, for a finite subset $S=\{v_1, \dots, v_k\}\subset N_\RR$, 
the cone of $S$ is a subset of $N_\RR$ of the form
$$\sigma:=\text{Cone}(S)=\left\{ \sum_{i=1}^k c_i v_i ~\Big|~ c_i\geq 0\right\}.$$
To each cone $\sigma=\text{Cone}(S)$, one can associate a \emph{dual cone} 
$\sigma^\vee \subset M_\RR$, where $M_\RR$ is a vector space 
dual to $N_\RR$,
defined as follows: 
$$\sigma^\vee= \{y \in M_\RR \mid  \left< y, v\right> \geq 0 \text{ for all } v\in \sigma\}.$$
One can easily see that $(\sigma^\vee)^\vee=\sigma$. 

A cone $\sigma_1:=\text{Cone}(S_1)$ is called a \emph{face} of another cone 
$\sigma_2:=\text{Cone}(S_2)$
if there is a vector $w\in \sigma_2^\vee$ and corresponding hyperplane 
$H_w:=\{ x\in N_\RR \mid \left<w, x \right> =0\}$ such that 
$\sigma_1= H_w \cap \sigma_2$ and $\sigma_2$ lie in 
the positive half space $H_w^+:=\{x\in N_\RR \mid \left<w, x\right> \geq 0\}$. 
Notice that $\sigma_2^\vee$ is 
a face of $\sigma_1^\vee$, if $\sigma_1$ is a face of $\sigma_2$.

The set of one dimensional face of a cone $\sigma$ is called the 
\emph{generating rays} of $\sigma$. 
A cone $\sigma$ is called \emph{strongly convex} if it does not 
contain a nontrivial linear subspace of $N_\RR$, and is called \emph{simplicial}
if its generating rays are linearly independent in $N_\RR$. 

Let $M$ be the lattice dual to $N$. The lattice points $S_\sigma:=\sigma^\vee\cap M$ 
form a finitely generated semigroup (see \cite[Proposition 1.2.17]{CLS}) 
which yields an algebraic variety 
$$U_\sigma:= \text{Spec}(\CC[S_\sigma]),$$
which is called the \emph{affine toric variety} associated to a cone $\sigma$.

\begin{remark}\label{rmk_torus_action_on_affine_toric}
Every affine toric variety $U_\sigma$ associated to a cone $\sigma\in N_\RR$ 
equips with the action of $(\CC^\ast)^n$, which is obtained by identifying
$(\CC^\ast)^n\cong \text{Spec}(\CC[M])$. Here we regard $M$ as the 
dual cone of $\{0\} \in N_\RR$. See \cite[Chapter 1]{CLS} or \cite[Section 1.3]{Ful-ITV}. 
\end{remark}

A \emph{fan} $\Sigma$ in $N_\RR$ is a finite collection of strongly convex 
cones such that
\begin{enumerate}
\item if $\sigma_2\in \Sigma$ and $\sigma_1$ is a face of $\sigma_2$, 
then $\sigma_1 \in \Sigma$;
\item if $\sigma_1, \sigma_2\in \Sigma$, then $\sigma_1 \cap \sigma_2$ is
a face of both $\sigma_1$ and $\sigma_2$. 
\end{enumerate}
One natural way of defining a fan in toric geometry is to use a full 
dimensional lattice polytope $Q \subset M_\RR$, see Remark \ref{rmk_simple_lattice_polytope}. 
To be more precise, let $\mathcal{F}(Q)=\{F_1, \dots, F_m\}$ be facets of $Q$ 
as in Subsection \ref{toric_orbifolds}.
For each $i$, let $\lambda_i\in N$ denote the vector such that 
\begin{equation}\label{eq_simple_polytope}
Q=\{ y\in M_\RR\mid \left< y, \lambda_i \right> + a_i \geq 0 \text{ for some}
 ~ a_i\in \ZZ~ \mbox{with}~ i=1, \dots, m \}.
\end{equation}
Observe that $\lambda_i$ in \eqref{eq_simple_polytope} represents the inward
normal vector of the facet 
$$F_i =\{ y\in M_\RR\mid \left< y, \lambda_i \right> + a_i = 0\} \cap Q$$
for $i=1, \ldots, m$. 
Let a face $E$ of $Q$ be given by $E=F_{i_1} \cap \dots \cap F_{i_k}$ for some facets $F_{i_1} , \dots, F_{i_k}$. 
Then we can associate a cone $\sigma_E$ to each face $E$ of $Q$
as follows:
$$\sigma_E = \text{Cone}\{\lambda_{i_1}, \dots, \lambda_{i_k}\}.$$

The fan $\Sigma_Q$ associated to a lattice polytope $Q$, 
which is called the \emph{normal fan} of $Q$, is the collection 
$\{\sigma_E\mid E \text{ is a face of } Q\}$ whose face structure 
naturally follows from the face structure of $Q$.  In general, we call a fan $\Sigma$ \emph{polytopal} 
if $\Sigma$ is a normal fan of a lattice polytope. 

Finally, the \emph{toric variety} $X_\Sigma$ corresponding to a fan $\Sigma$ is 
defined by taking disjoint union $\bigsqcup_{\sigma\in \Sigma}U_\sigma$ and 
the gluing is by rational maps determined by the generating rays of two cones 
$\sigma_1$ and $\sigma_2$ with nonempty intersections. In particular, the torus 
$(\CC^\ast)^n\cong U_{\{0\}} \subset U_\sigma$ of each affine toric variety 
is identified by the gluing, so yields the action of $(\CC^\ast)^n$ on 
$X_\Sigma$. 
Further details may be found in \cite[Section 1.4]{Ful-ITV}.

It is well-known from literature, such as \cite{CLS, Ew, Ful-ITV}, that 
a fan $\Sigma$ is \emph{polytopal} if and only if the corresponding toric variety 
$X_\Sigma$ is \emph{projective}, and $\Sigma$ is \emph{simplicial}, 
i.e., $\Sigma$ consists of simplicial cones, if and only if $X_\Sigma$ is an 
\emph{orbifold.}

Observe that if a polytopal fan $\Sigma$ is simplicial, then the corresponding 
polytope $Q$ is simple. Moreover, the primitive inward normal vector $\lambda_i$ 
of each facet $F_i\in \mathcal{F}(Q)$ satisfies the condition 
\eqref{eq_R-char_function_condition} since $Q$ is a simple lattice polytope.
Hence, such a fan $\Sigma$ naturally 
yields an $\mathcal{R}$-characteristic pair $(Q, \lambda)$, introduced in 
Subsection \ref{toric_orbifolds}, by defining $\lambda(F_i)$ to be the 
primitive inward normal vector of a facet $F_i\in \mathcal{F}(Q)$, 
see Remark \ref{rmk_simple_lattice_polytope}. 

%%%%%%%%%%%%%%%%%%%%%%%%%%%%%%%%%%%%%%%%%
%%%%%%%%%%%%%%%%%%%%%%%%%%%%%%%%%%%%%%%%%

In particular, if an $\mathcal{R}$-characteristic pair 
$(Q, \lambda)$ is defined from a polytopal fan $\Sigma$ as above, 
then one can show that there is a $T^n$-equivariant homeomorphism 
\begin{equation}\label{eq_toric_orb_iso}
X(Q, \lambda) \cong X_\Sigma,
\end{equation}
where $T^n$-action on $X(Q, \lambda)$ is 
the multiplication on the second factor (see \eqref{eq_def_of_toric_orb}) and $T^n$-action on $X_\Sigma$ 
is provided by regarding $T^n$ as a compact torus in $(\CC^\ast)^n$, 
see \cite[Chapter 12]{CLS} and \cite[Section 7]{DJ}.
Hence, one can apply theorems of Subsection \ref{toric_orbifolds} 
to a projective toric variety from a simplicial fan, i.e., a toric variety
from a normal fan of a simple lattice polytope. 

%To be more concrete, let $(Q, \lambda)$ be the $\mathcal{R}$-characteristic 
%function obtained from the normal fan $\Sigma$ of a simple lattice polytope. 
%Then, a vertex $v$ of $Q$ corresponds to an $n$-dimensional cone, say 
%$\sigma_v \in \Sigma$. If $v$ is the intersection of $n$ facets 
%$F_{i_1},\dots, F_{i_n}$, then 
%$$\sigma_v= \text{Cone}\{\lambda(F_{i_1}), \dots, \lambda(F_{i_n})\}.$$
%Hence, one can associate a number 
%$$g(\sigma_v):= g_Q(v)$$ 
%to each $n$-dimensional cone $\sigma_v \in \Sigma$, where 
%$g_Q(v)$ is defined in \eqref{eq_g_Q(v)}. 

%Whenever the cone $\sigma_1$ is a face of the cone $\sigma_2$, we write 
%$\sigma_1 \prec \sigma_2$. Suppose now that 
%$\tau \prec \sigma_v$, then $\tau$ corresponds to a 
%face $E$ of $Q$ which contains $v$ as a vertex. Hence, 
%we can also associate a number 
%$$g_\tau(\sigma_v) := g_E(v)$$ 
%to each pair of cones $(\tau \prec \sigma_v)$, 
%where $g_E(v)$ is defined in \eqref{eq_g_E(v)}. 
%When $\tau$ is also a face of another
%$n$-dimensional cone $\sigma_w$ corresponding to a vertex $w$ of $Q$, 
%then we associate to $(\tau \prec \sigma_w)$,
%the integer $g_\tau(\sigma_w)$. 

Recently, the authors of \cite{BSS} computed the integral cohomology 
ring of certain toric varieties associated to a normal fan 
$\Sigma$ of a simple polytope $Q$. The description was given in terms of the 
\emph{weighted Stanley--Reisner ring} $w\mathcal{SR}[\Sigma]$ of a 
simplicial fan $\Sigma$, which is a certain subring of the usual 
Stanley--Reisner ring $\mathcal{SR}[Q]=\ZZ[x_1, \dots, x_m]/\mathcal{I}$,
see \cite[(5.3)]{BSS}. This subring is determined 
by an \emph{integrality condition} which encodes the singularity of 
each fixed point. Then, assuming $H^{odd}(X_\Sigma;\ZZ)=0$,  
it was shown that there is an isomorphism between $H^\ast(X_\Sigma;\ZZ)$ and 
$w\mathcal{SR}[\Sigma]/ \mathcal{J}$, where $\mathcal{J}$ is an ideal
generated by the linear relations in \eqref{eq_linear_relations}, which
 is determined by the geometry of $\Sigma$. Hence, by the aid of 
Theorem \ref{thm_even_cells_no_p-torsion}, we get the following theorem 
which generalizes a result of \cite[Theorem 5.3]{BSS}.

%is an ideal determined by the geometry of $\Sigma$ defined by the linear relations in \eqref{eq_linear_relations}. 

%Next, we apply Theorem \ref{thm_even_cells_no_torsion}
%to projective toric varieties associated to a normal fan of a simple polytope 
%to get the following corollary. 

\begin{theorem}\label{cor_cohom_of_toric_orb_with_gcd_cond.}
Let $X_\Sigma$ be a projective toric variety associated to a normal 
fan $\Sigma$ of a simple polytope $Q$ with $m$ facets. 
If for each prime $p$ there is a retraction sequence $\{(B_i, E_i, v_i)\}_{i=1}^\ell$
of $Q$ such that $\gcd\{p, g_{E_i}(v_i)\}=1$, 
then the cohomology ring of  $X_\Sigma$  is
$$H^\ast(X_\Sigma;\Z) \cong w\mathcal{SR}[\Sigma]/\mathcal{J}\subseteq 
\ZZ[x_1, \dots, x_m]/\mathcal{I}+\mathcal{J},$$
where $\mathcal{I}$ is the Stanley--Reisner ideal of $Q$ and 
$\mathcal{J}$ is the ideal generated by linear relations 
\begin{equation}\label{eq_linear_relations}
\sum_{i=1}^m \langle \lambda_i, \mathbf{e}_j\rangle x_i=0, 
\quad  j=1, \dots, n,
\end{equation}
where $\deg x_i=2$, $\mathbf{e}_j$ denotes the $j$-th standard unit vector in $\Z^n$
and $\lambda_i$ is defined by \eqref{eq_simple_polytope}.
\end{theorem}

\subsection{Torus orbifolds}\label{subsec_torus_orb}
A \emph{torus orbifold} is a $2n$-dimensional closed orbifold 
with an action of the $n$-dimensional real torus with non-empty 
fixed points. Introduced by Hattori and Masuda \cite{HM}, they are a far reaching generalization, beyond the spaces $X(Q, \lambda)$, of singular toric varieties having orbifold singularities. 
In this subsection, we recall the definition of locally standard torus 
orbifolds and apply Theorem \ref{thm_even_cells_no_p-torsion} and \ref{thm_even_cells_no_torsion}  on this class.

The faces of a manifold with corners can be defined following \cite[Section 6]{Dav}.  
A manifold with corners is called \emph{nice} if a codimension $2$ face is a 
connected component of the intersection of two codimension $1$ faces. 
Let $P$ be an $n$-dimensional nice manifold with corners 
and $\mathcal{F}(P)=\{F_1, \dots, F_m\}$ the codimension-$1$ 
faces of $P$. One can define $\mathcal{R}$-characteristic function
$\lambda$ on $P$ satisfying the condition \eqref{eq_R-char_function_condition}. 
In Figure \ref{fig-eg2}, we give some examples of nice manifolds with corners and
$\mathcal{R}$-characteristic functions on them. 

A torus orbifold $X_P(\lambda)$ from a pair $(P, \lambda)$, where 
$P$ is a nice manifold with corners and $\lambda$ is an 
$\mathcal{R}$-characteristic function on $P$, can be defined similarly to 
the construction of a toric orbifold from an $\mathcal{R}$-characteristic 
pair, see for example \cite{KMZ, MMP, MP} and \eqref{eq_def_of_toric_orb} in
Subsection \ref{toric_orbifolds}. 
Moreover, all the notation of the local structures, such as $g_Q(v)$ and 
$g_E(v)$ of a toric orbifold in 
Subsection \ref{toric_orbifolds} naturally extends to a torus orbifold $X_P(\lambda)$. 
We remark that a certain non-convex
manifold with corners and corresponding torus manifolds are studied
in \cite{PS2} .

\begin{example}\label{eg02}
In this example we exhibit two 3-dimensional manifolds with corners and 
an $\mathcal{R}$-characteristic function on each.
The figure (a) is a cylinder over an eye-shape, which is not a simple polytope.
In figure (b),  the intersection of the facets $ABCDEF$ and $BCHEFG$ is the
disjoint union of edges $BC$ and $EF$, so it is also not a simple polytope. 

\begin{figure}[ht]
        \centerline{
           \scalebox{.56}{
            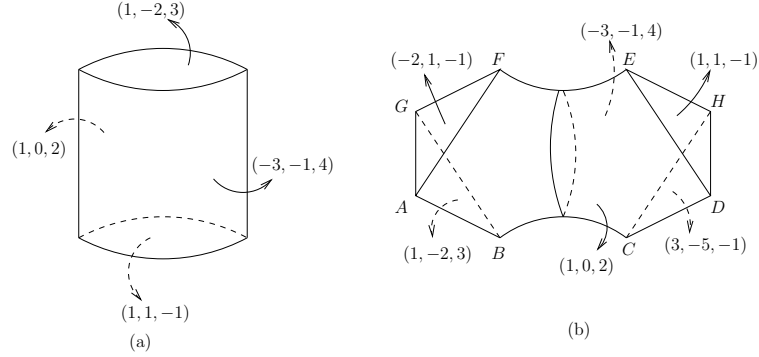
            }
          }
 \caption{Two $\mathcal{R}$-characteristic functions.}
 \label{fig-eg2}
 \end{figure} 
 \end{example}

\begin{figure}[ht]
        \centerline{
           \scalebox{.50}{
            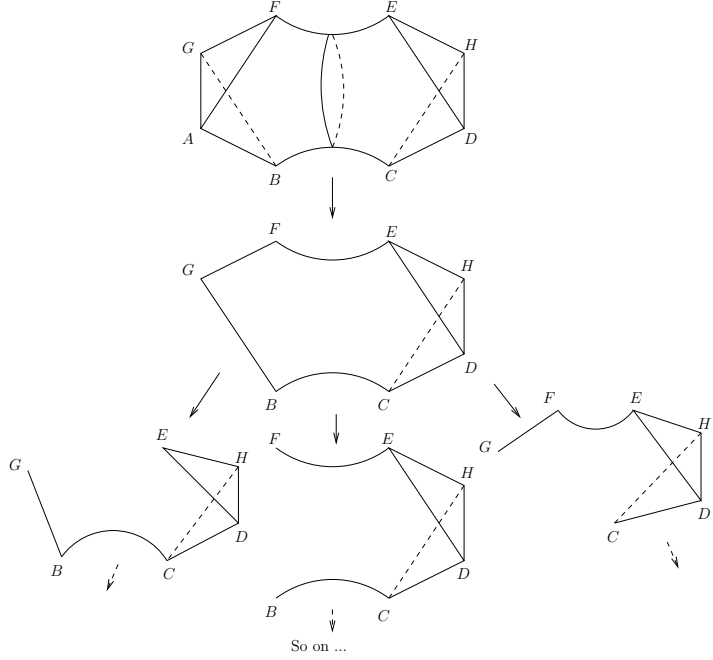
            }
          }
 \caption{Some retraction sequences of a nice manifold with corners.}
 \label{fig-eg3}
 \end{figure} 

We remark that a retraction sequence may exist for a nice manifold
with corners which is not necessarily a simple polytope; see Figure
\ref{fig-eg3} for an example. We also note that if  $\{(B_i, E_i, v_i)\}_{i=1}^\ell$
is a retraction sequence of a nice manifold with corners $P$ and $X_P(\lambda)$
is a torus orbifold, then one can construct a building sequence for
$X_P(\lambda)$ analogous to Proposition \ref{ret_build}.
Now we can apply Theorem \ref{thm_even_cells_no_torsion} to these class of orbifolds.

\begin{corollary}
Let $X_P(\lambda)$ be a torus orbifold over the nice manifold with corners
$P$. If for each prime $p$ there is a retraction $\{(B_i, E_i, v_i)\}_{i=1}^\ell$
of $P$ such that $\gcd\{p, g_{E_i}(v_i)\}=1$, then $H^\ast(X_P(\lambda); \ZZ)$ 
is torsion free and concentrated in even degrees. 
\end{corollary}

\subsection{A $\mathbf{q}$-cell structure for Weighted Grassmannians}\label{subsec_wGr}
In this subsection, we  introduce briefly the construction of a weighted 
Grassmannian $\mathbf{w}Gr(d,n)$ \cite{CR-wgr} and its $\mathbf{q}$-CW structure
\cite[Section 2]{AM}. In addition, we show an application of Theorem
\ref{thm_even_cells_no_p-torsion} to certain weighted Grassmanians. 

We first consider the action of the general linear group $GL_n(\CC)$ on 
the $d$-th exterior product $\wedge^d\CC^n$ of $\CC^n$ induced
from the canonical action of $GL_n(\CC)$ on $\CC^n$. It naturally 
induces the action of $GL_n(\CC)$ on the projective space 
$\mathbb{P}(\wedge^d \CC^n)$. The standard Pl\"ucker embedding of
 the Grassmannian $Gr(d,n)$ into $\mathbb{P}(\wedge^d \CC^n)$
is given by the $GL_n(\CC)$ orbit of $[e_1\wedge \dots \wedge e_d]$. 

Now, we consider 
$$aPl(d,n)^\ast :=  GL_n(\CC)\cdot (\CC e_1\wedge \cdots \wedge e_d)    
\setminus \{\mathbf 0\},$$ 
where 
$GL_n(\CC)\cdot (\CC e_1\wedge \cdots \wedge e_d)$ 
denotes the $GL_n(\CC)$-orbit of the line 
$\CC e_1\wedge \cdots \wedge e_d \subset \wedge^d\CC^n$ generated by
$e_1\wedge \cdots \wedge e_d$, 
and $\mathbf{0}$ denotes the origin in $\wedge^d\CC^n$. 
Next, for the maximal torus $T$ in $GL_n(\CC)$, we consider 
\begin{equation}\label{eq_group_K}
K:=T\times \CC^\ast \cong (\CC^\ast)^{n+1}
\end{equation}
and its action on 
$\wedge^d \CC^n$, where $T$ acts naturally as a subgroup of 
$GL_n(\CC)$ and $\C^\ast$ acts by the scalar multiplication on the 
vector space $\wedge^d \CC^n$. Then, $aPl(d, n)$ is equipped with a  
$K$-action induced from the $K$-action on $\wedge^d \CC^n$. 

Finally, we choose a vector 
$\mathbf{w}:=(w_1, \dots, w_n, r) \in \ZZ^n_{\geq 0} \oplus \ZZ_{>0}$ and 
we define 
$$\mathbf{w}D:=\{(t^{w_1}, \dots, t^{w_n}, t^r)\in K \mid t\in \CC^\ast\},$$
the weighted diagonal subgroup of $K$ with respect to the vector 
$\mathbf{w}$. 
Then, the weighted Grassmannian $\mathbf{w}Gr(d,n)$ is defined by the quotient 
$$\mathbf{w}Gr(d,n):=aPl(d,n)^\ast / \mathbf{w}D.$$
Notice that $\mathbf{w}Gr(d, n)$ equips with the action of residual torus $K/\mathbf{w}D$. 

\begin{example}
\begin{enumerate}
\item Choose $\mathbf{w}=(0, \dots, 0, 1)\in \ZZ^n_{\geq 0}\oplus \ZZ_{>0}$, 
which  defines $\mathbf{w}D=(1, \dots, 1, t)\subset (\CC^\ast)^{n+1}$. 
The quotient space $aPl(d, n)^\ast/\mathbf{w}D$ is the image of 
the standard Pl\"ucker embedding of Grassmannian 
$Gr(d,n)$ into $\mathbb{P}(\wedge^d\CC^n)$. 
\item Consider the case when $d=1$ and choose an arbitrary weighted vector 
$\mathbf{w}=(w_1, \dots, w_n, r)$. Then, $GL_n(\CC)$-orbit of the line $\CC e_1$
is isomorphic to $\CC^n$. Hence, $aPl(1, n)^\ast=\CC^n \setminus \{\mathbf 0\}$. 
Notice that the non-zero scalar multiple of the first column of $GL_n(\CC)$ generates the space
 $\CC^n \setminus \{\mathbf 0\}$. Hence, the action of $\mathbf{w}D$ on 
$\CC^n \setminus \{ \mathbf{0}\}$ is given by 
$$(t^{w_1}, \dots, t^{w_n}, t^r) \cdot (z_1, \dots, z_n)= 
(t^{w_1+r}z_1, \dots, t^{w_n+r}z_1),$$
which leads us to the weighted projective space $\CP^n_{(w_1+r, \dots, w_n+r)}$.  
\end{enumerate}
\end{example}

%In order to study a $\mathbf{q}$-CW structure on $\mathbf{w}Gr(d,n)$, 
%we  introduce more notation and review the classical cell 
%decomposition of $Gr(d, n)$. 
One natural way to describe a $\mathbf{q}$-CW complex structure on $\mathbf{w}Gr(d,n)$
is to use the classical Schubert cell decomposition of $Gr(d, n)$.
Let $\alpha=\{i_1  < \dots< i_d\}$ be a subset of $[n]:=\{1, \dots, n\}$. 
Then, a $(d\times n)$-matrix of the following row echelon form 
represents a cell $e_\alpha$ with dimension $\sum_{k=1}^d (i_k-k)$ of 
$Gr(d, n)$;
$$\begin{blockarray}{ccccccccccccccc}
\begin{block}{(ccccccccccccccc)}
%%\ast&\cdots&\ast&1&0&\cdots&&  &&&	\\
%%\ast&\cdots&\ast&0&\ast&\cdots &\ast&1&0&\cdots&	\\
%%&&&\vdots&&&&\vdots &&&\\
%%\ast&\cdots&\ast&0&\ast&\cdots&\ast&0&\ast&\cdots&\ast
\ast &\cdots &\ast &1&0&\cdots &&&&&&&&\cdots &0 \\
\ast&\cdots&\ast&0&\ast&\cdots&\ast&1&0&\cdots&&&&\cdots&0 \\
\ast&\cdots&\ast&0&\ast&\cdots&\ast&0&&&&&&& \\
\vdots&&\vdots&\vdots&&&&\vdots&&&\ddots&&&&\vdots \\
\ast&\cdots&\ast&0&\ast&\cdots&\ast&0&\ast&\cdots&\ast&1&0&\cdots&0\\
\end{block}
&&&\overset{\uparrow}{i_1\text{-th}}&&&&\overset{\uparrow}{i_2\text{-th}}&&\cdots&&\overset{\uparrow}{i_d\text{-th}}&&&
\end{blockarray}$$
The call $e_\alpha$ is called the \emph{Schubert cell} corresponding to 
the subset $\alpha$. In the literature, a {\em Young diagram}
corresponding to $(i_1-1, \dots, i_d-d)$ is used to describe $e_\alpha$. 
For example, if $n=10$, $d=3$ and $\alpha=\{2<5<9\}$, the Schubert cell
$e_{ \{2,5,9\}}$ and the corresponding Young diagram are described below in  
 Figure \ref{fig_Schubert_cell_and_Young_diag}. 
\begin{figure}[h]
\begin{tikzpicture}
\node at (0,0) {$\begin{bmatrix}
\ast&1&0&0&0&0&0&0&0&0\\
\ast&0&\ast&\ast&1&0&0&0&0&0\\
\ast&0&\ast&\ast&0&\ast&\ast&\ast&1&0
\end{bmatrix}
$};
\draw[<->] (3,0)--(4,0);
\begin{scope}[scale=0.4, xshift=110mm, yshift=-15mm]
\draw (0,0)--(6,0)--(6,1)--(3,1)--(3,2)--(1,2)--(1,3)--(0,3)--cycle;
\draw (1,0)--(1,2);
\draw (2,0)--(2,2);
\draw (3,0)--(3,2);
\draw (5,0)--(5,1);
\draw (4,0)--(4,1);
\draw (0,1)--(3,1);
\draw (0,2)--(1,2);
\end{scope}
\end{tikzpicture}
\caption{A Schubert cell and Young diagram.}
\label{fig_Schubert_cell_and_Young_diag}
\end{figure}

\begin{remark}
\begin{enumerate}
\item The Young diagram as defined in \cite{Ful-YT} has the weakly decreasing 
number of boxes in each rows. In this article, we use the 
\emph{weakly increasing number of boxes} in each rows as a convention. 
\item The collection of Young diagram has an obvious partial order, namely
two Young diagrams $\square_1 \subseteq \square_2$ if and only if 
$\square_1$ ``fits" inside $\square_2$. 
\end{enumerate}
\end{remark}

Finally, the Schubert cell decomposition is given by 
\begin{equation}\label{eq_Schubert_cell_decomp}
Gr(d, n)=\bigsqcup_{\alpha\in [n]} e_\alpha.
\end{equation}
The complex dimension of the cell $e_\alpha$ is exactly the number of 
boxes in the corresponding Young diagram, and the information about 
attaching maps can be obtained from the partial order 
given by the inclusion relations among the Young diagrams. 

In order to describe a $\mathbf{q}$-CW structure for $\mathbf{w}Gr(d,n)$, 
we note that the Schubert cell decomposition \eqref{eq_Schubert_cell_decomp}
together with the orbit map $\pi \colon aPl(d,n)^\ast \to Gr(d,n)$
induces the following cell decomposition;
\begin{equation}\label{eq_cell_decomp_aPl}
aPl(d,n)^\ast=\bigsqcup_{\alpha\in [n]}\pi^{-1}(e_\alpha).
\end{equation}
Next, we recall from  \cite[Subsection 2.2]{AM} 
that each cell $\pi^{-1}(e_\alpha)$ is $K$-invariant, where $K$
is defined in \eqref{eq_group_K}. Hence, the cell decomposition 
\eqref{eq_cell_decomp_aPl} descends to the cell decomposition 
of the weighted Grassmannian.  
\begin{proposition}(\cite[Proposition 2.3]{AM})
$$\mathbf{w}Gr(d,n) = \bigsqcup_{\alpha \in [n]} e_\alpha/G_\alpha,$$
where $G_\alpha=\{ (t^{w_1}, \dots, t^{w_n}, t^r)\in K \mid t\in \CC^\ast, ~
t^{w_\alpha}=1, ~ w_\alpha= r+\sum_{i\in \alpha}w_i\}$. 
\end{proposition}

\begin{figure}
\begin{tikzpicture}
\begin{scope}[scale=0.4]
\node at (0.5,-2) {$\varnothing$};
\draw (0,0)--(0,1)--(1,1)--(1,0)--cycle; %e_13
\draw (3,2)--(4,2)--(4,3)--(2,3)--(2,2)--(3,2)--(3,3); %e_14
\draw (-2,3)--(-2,2)--(-1,2)--(-1,4)--(-2,4)--(-2,3)--(-1,3); %e_23
\draw (-0.5, 5+1)--(1.5, 5+1)--(1.5, 6+1)--(0.5, 6+1)--(0.5, 7+1)--(-0.5, 7+1)--cycle; %e_24
\draw (-0.5, 6+1)--(0.5, 6+1)--(0.5, 5+1); 
\draw (-0.5, 8+2)--(1.5, 8+2)--(1.5, 10+2)--(-0.5, 10+2)--cycle; %e_34
\draw (0.5, 8+2)--(0.5, 10+2); \draw (-0.5, 9+2)--(1.5, 9+2);

\draw (0.5,-1.1)--(0.5, -0.5);
\draw (-0.2 ,1.2)--(-0.8,1.8);
\draw (1.2, 1.2)--(1.8, 1.8);
\draw (-1.5,4.3)--(-0.7,5.8);
\draw (3, 3.3)--(1.7, 5.8);
\draw (0.5,9.7)--(0.5, 8.3);

\end{scope}

\begin{scope}[xshift=5cm, scale=0.4]
\node at (0,-2) {$e_{12}/G_{12}$};
\node at (0,0.5) {$e_{13}/G_{13}$};
\node at (-3, 3) {$e_{23}/G_{23}$};
\node at (3, 3){$e_{14}/G_{14}$};
\node at (0, 7) {$e_{24}/G_{24}$};
\node at (0, 11) {$e_{34}/G_{34}$};

\draw (0,-1.1)--(0, -0.5);
\draw (-1.2 ,1.2)--(-1.8,2);
\draw (1.2, 1.2)--(1.8, 2);
\draw (-2,4)--(-1.2,5.8);
\draw (2, 4)--(1.2, 5.8);
\draw (0,9.7)--(0, 8.3);
\end{scope}

\end{tikzpicture}
\caption{Lattice of Young diagrams and $\mathbf{q}$-CW structure of 
$\mathbf{w}Gr(2,4)$.}
\label{fig_Young_lattice_Gr(2,4)}
\end{figure}
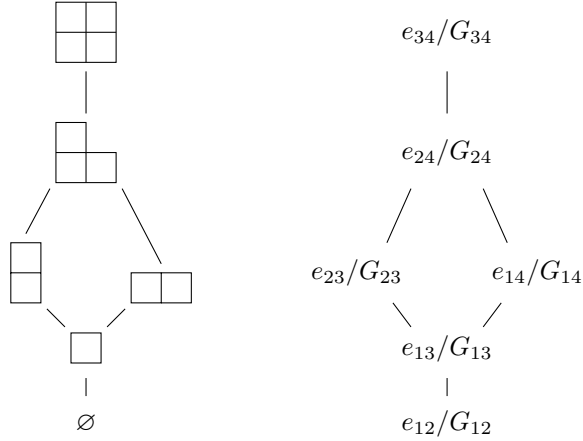

\begin{example}
Consider the weighted Grassmannian $\mathbf{w}Gr(2,4)$ 
determined by the weight 
$\mathbf{w}=(w_1, \dots, w_4, r)\in \ZZ_{\geq0}^4 \oplus \ZZ_{>0}$. 
The lattice structure of Young diagrams and $\mathbf{q}$-CW structure 
are described in Figure \ref{fig_Young_lattice_Gr(2,4)}. 
Here,  
$$G_{ij}=\{(t^{w_1}, t^{w_2}, t^{w_3}, t^{w_4}, t^r) \mid 
t\in \CC^\ast,~t^{r+w_i+w_j}=1\}.$$
For example, when $\mathbf{w}=(1,1,1,1,1)$, 
for each $\{i,j\}\subset [4]$, $G_{ij}\cong \ZZ/3\ZZ$. 
Hence, from Theorem \ref{thm_even_cells_no_p-torsion}, we conclude that 
$H^\ast(\mathbf{w}Gr(2,4);\ZZ)$ has no $p$-torsion unless $p=3$.
In general, one can conclude from this observation that the  
cohomology ring of the weighted Grassmannian $\mathbf{w}Gr(d, n)$, 
determined by the weighted vector of the form 
$\mathbf{w}=(1,\dots, 1,1)\in \ZZ_{\geq 0}^n \oplus \ZZ_{>0}$,
has no $p$-torsion unless $p=d+1$. 
\end{example}

\section{An extension to more general $\mathbf{q}$-CW complexes}
\label{sec_cells_in_every_dim}

In this section, we discuss general $\mathbf{q}$-CW
complexes which do not necessarily consist of even dimensional $\mathbf{q}$-cells.
Up to this point, the degree, in the sense of Definition \ref{def_degree} has
played no role. In fact, let $\{(Y_i, 0_i)\}_{i=1}^\ell$ be a building sequence
 satisfying the assumption of 
Theorem \ref{thm_even_cells_no_torsion}. The attaching map
$$\phi_i \colon S^{2k_i-1}/G_i \to Y_{i-1}$$
induces the \emph{degree} for each $i=1, \dots, \ell$. 
However, the induced homology map 
$$(\phi_i)_\ast \colon H_{2k_i-1}(S^{2k_i-1}/G_i;\ZZ) \to
H_{2k_i-1}(Y_{i-1};\ZZ)$$
is a zero map because $H_{2k_i-1}(Y_{i-1};\ZZ)$ is trivial by Theorem \ref{thm_even_cells_no_torsion}.

If a building sequence $\{(Y_i, 0_i)\}_{i=1}^\ell$ for a $\mathbf{q}$-CW complex $X$ 
has odd dimensional cells, a non-trivial attaching map is possible; this leads to the following theorem. 

\begin{theorem}\label{no_p-torsion_thm}
Let $X$ be a $\mathbf{q}$-CW complex, $p$ a prime number, and
$\{(Y_i, 0_i)\}_{i=1}^\ell$ a building sequence for $X$ such that $\gcd\{p, |G_i|\}=1$
for all $i$ with $e^{k_{i}}/G_i =Y_i \setminus Y_{i-1}$. If the degree of
each attaching map 
\begin{equation}\label{eq_attaching_map}
S^{k_{i}-1}/G_i \to Y_{i-1}
\end{equation}
as defined in Section \ref{sec_q-CW_complex},
is coprime to $p$ or zero, then $H_{\ast}(X; \Z)$
has no $p$-torsion.
\end{theorem}
\begin{proof}
Since $Y_{i}$ is obtained by attaching a $\mathbf{q}$-cell
 $\bar{e}^{k_i}/G_{i}$ to $Y_{i-1}$, we have the cofibration:
  \begin{equation}
\begin{tikzcd}
S^{k_i-1}/G_i \arrow{r} & Y_{i -1} \arrow[hookrightarrow]{r} & Y_i
\end{tikzcd}
%\xymatrix{S^{k_i-1}/G_i \ar[r]& Y_{i -1} \ar@{^{(}->}[r] & Y_i },
\end{equation} 
which yields the long exact sequence: 
\begin{equation}\label{eq_les_near_top_homology}
\begin{tikzcd}[row sep=tiny, column sep=small]
\cdots \rar&   
\widetilde{H}_{j}(S^{k_i-1}/G_i) \rar& 
H_{j}(Y_{i-1}) \arrow{r} & 
H_{j}(Y_{i}) \arrow{r}  &   
\widetilde{H}_{j-1}(S^{k_i-1}/G_i) \arrow{r} & 
\cdots .
\end{tikzcd}
\end{equation}
The proof is similar to that of Theorem \ref{thm_even_cells_no_p-torsion}, 
except in the case when $j$ corresponds to the top two dimensions, $k_i$
and $k_i-1$. In this case, $\widetilde{H}_{k_i}(S^{k_i-1}/G_i)=0$ for
dimensional reasons. By the same induction hypothesis as Theorem \ref{thm_even_cells_no_p-torsion}, we assume that $H_{k_i}(Y_{i-1})$ and
$H_{k_i-1}(Y_{i-1})$ have no $p$-torsion. 
Hence, the sequence \eqref{eq_les_near_top_homology} becomes 
\begin{equation}\label{eq_les_near_top_homology(2)}
\begin{tikzcd}[row sep=tiny]
 0 \arrow{r} & \ZZ^r \oplus K_1 \arrow{r}{f}  & H_{k_i}(Y_{i}) \arrow{r}{g} & \widetilde{H}_{k_i-1}(S^{k_i-1}/G_i) &\\ 
 \arrow{r}{(\phi_i)_{\ast}}  & \ZZ^s\oplus K_2 \arrow{r}{f'}  & H_{k_i-1}(Y_{i})
 \arrow{r}{g'} & \text{$|G_i|$-torsion} \arrow{r} & \cdots 
\end{tikzcd}
\end{equation}
for some $r, s \in \NN\cup \{0\}$ and some $p$-torsion free  finite groups $K_1$ and $K_2$. 

In order to show both $H_{k_i}(Y_{i})$ and $H_{k_i-1}(Y_{i})$ have no $p$-torsion, 
we need to consider the following three cases prescribed by Proposition \ref{prop_homology_orb_lens_sp}. 
\begin{enumerate}
\item If $\widetilde{H}_{k_i-1}(S^{k_i-1}/G_i)=0$, then 
$H_{k_i}(Y_{i})$ is isomorphic to $\ZZ^r \oplus K_1$, hence it has no $p$-torsion. 
Next suppose that $H_{k_i-1}(Y_{i})$ has a $p$-torsion element, say $x$, then 
$g'(x)=0$ because $\gcd\{ p, |G_i|\}=1$. Hence, $x\in \ker g' = \im f'$.
Since $\widetilde{H}_{k_i-1}(S^{k_i-1}/G_i)=0$, the map $f'$ is injective.
It contradicts the $p$-torsion freeness of $K_2$. 
\item If $\widetilde{H}_{k_i-1}(S^{k_i-1}/G_i)$ is a non-trivial $|G_i|$-torsion, then an argument similar to that above shows that $H_{k_i}(Y_{i})$ has no $p$-torsion. 
Next, consider the following exact sequence:
$$
\begin{tikzcd}
\text{ a $|G_i|$-torsion} \arrow{r}{(\phi_i)_{\ast}} &
\ZZ^s\oplus K_2 \arrow{r}{f'} & 
H_{k_i-1}(Y_{i}) \arrow{r}{g'} & 
\text{a $|G_i|$-torsion}.
\end{tikzcd}
$$
Suppose that $H_{k_i-1}(Y_{i})$ has a $p$-torsion element $x$. Then the assumption 
$\gcd\{p, |G_i|\}=1$ implies $g'(x)=0$. Hence, we have the following nontrivial 
preimage
\begin{equation}\label{eq_preimage}
(f')^{-1}(x)=\{(a,b)\in \ZZ^s \oplus K_2 \mid f'(a,b)= x \}.
\end{equation} 
Notice that the first coordinate of an element in \eqref{eq_preimage} cannot be $0$, 
because $K_2$ is a $p$-torsion free finite group and for a homomorphism $h: G \to H$
of finite groups, the order of $h(a)$ divides the order of $a \in G$. Hence, there exists
at least one nontrivial element $a \in \ZZ^s$ such that $f'(a,b)=x$. 
In particular, since $x$ is a $p$-torsion element, 
$$\{n\cdot p\cdot (a,b)\mid n\in \ZZ\} \subseteq \ker f' =\im (\phi_i)_{\ast},$$
which contradicts the finiteness of a $|G_i|$-torsion.  
\item Finally, we consider the case when $\widetilde{H}_{k_i-1}(S^{k_i-1}/G_i)=\ZZ$ and
break the exact sequence \eqref{eq_les_near_top_homology(2)} into two parts as follows:
\begin{align}
 0 \longrightarrow ~ &\ZZ^r \oplus K_1
 \overset{f}{\longrightarrow}  H_{k_i}(Y_{i})   \overset{g}{\longrightarrow}  \im g \longrightarrow 0, 
 \label{eq_first_ex_seq_after_breaking}\\
0 \longrightarrow \coker g \overset{(\phi_i)_{\ast}}{\longrightarrow}  ~~
& \ZZ^s\oplus K_2 \overset{f'}{\longrightarrow}  
 H_{k_i-1}(Y_{i}) \overset{g'}{\longrightarrow} \text{a $|G_i|$-torsion}. \label{eq_second_ex_seq_after_breaking}
\end{align}
Since $\im g \cong n\ZZ$ for some $n\in \NN \cup \{0\}$, applying the same argument as in Case (1)
to \eqref{eq_first_ex_seq_after_breaking}  shows that $H_{k_i}(Y_{i})$ has no $p$-torsion elements. 

%Finally, if $g$ is a non-zero map, then $\coker g$ is finite. Hence the same argument
%as in the second part of the case (2) shows that $H_{k_i-1}(Y_{i})$ is $p$-torsion free. 
%When the map $g$ a zero map, then $\coker g =\ZZ$, where the map 
%$\delta$ sends the generator $1\in \ZZ$ to $((b_1, \dots, b_s), x)\in \ZZ^s\oplus K_2$, 
%where $\gcd\{b_1, \ldots, b_s\}$ is the degree of the attaching map 
% \eqref{eq_attaching_map}, defined in the end of Section \ref{sec_q-CW_complex}.
% If $H_{k_i-1}(Y_{i})$ has a $p$-torsion element $a$, then either $a$ is the image
% of an element from $0 \oplus K_2$ or $a$ is an image of $((b_1, \dots, b_s), x)\in 
% \ZZ^s\oplus K_2$. Since $K_2$ has no $p$-torsion element, then first case is not
% possible. For the second case, since the sequence is exact, then $p$ is a divisor 
% of $\gcd\{b_1, \ldots, b_s\}$ which is not possible by the assumption that degree
% of the attaching map is co-prime to $p$. 

In order to show that $H_{k_i-1}(Y_{i})$ has no $p$-torsion elements,
 we first assume that $g$ is a non-zero map. 
Then $\coker g$ is finite, which implies that $H_{k_i-1}(Y_{i})$ is $p$-torsion free
by the same argument as in the second part of the case (2).

When the map $g$ is a zero map, then coker $g = \ZZ $. Now 
$$(\phi_i)_{\ast} \colon \widetilde{H}_{k_i-1}(S^{k_i-1}/G_i) \to H_{k_i-1}(Y_{i-1})$$
is the map induced from the attaching map  
$\phi_i \colon S^{k_i-1}/G_i \to Y_{i-1}$. 
Hence, the image of $1\in \ZZ\cong  \widetilde{H}_{k_i-1}(S^{k_i-1}/G_i)$ via $(\phi_i)_{\ast}$ 
determines the degree (see Definition \ref{def_degree}) of the attaching map, which is 
coprime to $p$ by the assumption.  

Now, we suppose that there is a $p$-torsion element $a\in  H_{k_i-1}(Y_{i})$. Since
$g'(a)=0$, there is an element, say $((c_1, \dots, c_s), y)\in \ZZ^s\oplus K_2$, such that 
$f'((c_1, \dots, c_s), y)=a$. We notice that 
%the subset
%\begin{equation}\label{eq_subgroup_mapsto_torsion}
%\{((c_1, \cdots, c_s),z)\mid z\in K_2\}\subset \ZZ^s\oplus K_2
%\end{equation}
%with
 $(c_1, \dots, c_s)\neq (0, \dots, 0)$, 
because $f'((0, \dots, 0), y)$ cannot be a $p$-torsion element in $H_{k_i-1}(Y_{i})$.
Observe that 
$$f'(p\cdot (c_1, \dots, c_s), p\cdot y)= p\cdot a=0\in H_{k_i-1}(Y_{i}).$$
Hence, there exists $m\in \ZZ$ such that 
$$(\phi_i)_{\ast}(m) = m \cdot (\phi_i)_{\ast}(1) = m \cdot ((d_1, \dots, d_s), x)=(p \cdot (c_1, \dots, c_s), p \cdot y).$$
Since the degree 
$$\deg (\phi_i \colon S^{k_i-1}/G_i \to Y_{i-1}) = \gcd \{d_i \mid 1\leq i\leq s,~d_i\neq 0\}$$ 
of the attaching map is coprime to $p$ by the assumption,
we conclude that $p$ divides $m$, i.e., $m=p\cdot m'$ for some $m'\in \ZZ \setminus \{0\}$. 
So,  $y \equiv m' \cdot x$ (mod $p$) and $c_j=m' \cdot d_j$ for $j=1, \ldots, s$. Thus, we have  
\begin{align*}
a=f'((c_1, \dots, c_s), m'\cdot x)
=m' \cdot f'((d_1, \dots, d_s), x )
=m' \cdot (f'\circ (\phi_i)_{\ast})(1)=0
\end{align*}
which is a contradiction. Hence, we conclude that $H_{k_i-1}(Y_{i})$ has no 
$p$-torsion elements.
\end{enumerate}

\end{proof}

\begin{corollary}\label{no_torsion_coro}
If $X$ is a $\mathbf{q}$-CW complex satisfying the assumption of Theorem 
\ref{no_p-torsion_thm} for each prime $p$, then $H_{\ast}(X; \Z)$ has no torsion. 
\end{corollary}

\begin{corollary}
Let $\{(Y_i, 0_i)\}_{i=1}^\ell$ be a building sequence for a $\mathbf{q}$-CW
complex $X$ such that $H_{\ast}(Y_{i -1}; \ZZ)$ has $p$-torsion and
 $\gcd\{p, |G_j|\}=1$ with $e^{k_i}/G_i =Y_i \setminus Y_{i-1}$ for $j > i$.
 Then $H_{\ast}(Y_{i}; \ZZ)$ has $p$-torsion, and hence $H_{\ast}(X; \ZZ)$
 has $p$-torsion.
\end{corollary}
\begin{proof}
This follows from \eqref{eq_les_near_top_homology}.
\end{proof}

\renewcommand{\refname}{References}

\bibliographystyle{alpha}
%\bibliography{bibliography.bib}

\end{document}

%% file: torus4.pdf_t
\begin{picture}(0,0)%
\includegraphics{torus4.pdf}%
\end{picture}%
\setlength{\unitlength}{4144sp}%
\begingroup\makeatletter\ifx\SetFigFont\undefined%
\gdef\SetFigFont#1#2#3#4#5{%
  \reset@font\fontsize{#1}{#2pt}%
  \fontfamily{#3}\fontseries{#4}\fontshape{#5}%
  \selectfont}%
\fi\endgroup%
\begin{picture}(10287,4728)(439,-4135)
\put(811,-1276){\makebox(0,0)[lb]{\smash{{\SetFigFont{12}{14.4}{\rmdefault}{\mddefault}{\updefault}{\color[rgb]{0,0,0}$B_1$}%
}}}}
\put(2836,-1276){\makebox(0,0)[lb]{\smash{{\SetFigFont{12}{14.4}{\rmdefault}{\mddefault}{\updefault}{\color[rgb]{0,0,0}$B_2$}%
}}}}
\put(1756,434){\makebox(0,0)[lb]{\smash{{\SetFigFont{12}{14.4}{\rmdefault}{\mddefault}{\updefault}{\color[rgb]{0,0,0}$v_{235}$}%
}}}}
\put(3826,-466){\makebox(0,0)[lb]{\smash{{\SetFigFont{12}{14.4}{\rmdefault}{\mddefault}{\updefault}{\color[rgb]{0,0,0}$v_{236}$}%
}}}}
\put(4861,434){\makebox(0,0)[lb]{\smash{{\SetFigFont{12}{14.4}{\rmdefault}{\mddefault}{\updefault}{\color[rgb]{0,0,0}$v_{345}$}%
}}}}
\put(4816,-1276){\makebox(0,0)[lb]{\smash{{\SetFigFont{12}{14.4}{\rmdefault}{\mddefault}{\updefault}{\color[rgb]{0,0,0}$B_3$}%
}}}}
\put(6121,-1276){\makebox(0,0)[lb]{\smash{{\SetFigFont{12}{14.4}{\rmdefault}{\mddefault}{\updefault}{\color[rgb]{0,0,0}$B_4$}%
}}}}
\put(7516,-466){\makebox(0,0)[lb]{\smash{{\SetFigFont{12}{14.4}{\rmdefault}{\mddefault}{\updefault}{\color[rgb]{0,0,0}$v_{346}$}%
}}}}
\put(7471,-1276){\makebox(0,0)[lb]{\smash{{\SetFigFont{12}{14.4}{\rmdefault}{\mddefault}{\updefault}{\color[rgb]{0,0,0}$B_5$}%
}}}}
\put(8911,-1276){\makebox(0,0)[lb]{\smash{{\SetFigFont{12}{14.4}{\rmdefault}{\mddefault}{\updefault}{\color[rgb]{0,0,0}$B_6$}%
}}}}
\put(9361,-916){\makebox(0,0)[lb]{\smash{{\SetFigFont{12}{14.4}{\rmdefault}{\mddefault}{\updefault}{\color[rgb]{0,0,0}$v_{126}$}%
}}}}
\put(9856,-16){\makebox(0,0)[lb]{\smash{{\SetFigFont{12}{14.4}{\rmdefault}{\mddefault}{\updefault}{\color[rgb]{0,0,0}$v_{145}$}%
}}}}
\put(9856,-1231){\makebox(0,0)[lb]{\smash{{\SetFigFont{12}{14.4}{\rmdefault}{\mddefault}{\updefault}{\color[rgb]{0,0,0}$B_7$}%
}}}}
\put(4411,-2851){\makebox(0,0)[lb]{\smash{{\SetFigFont{12}{14.4}{\rmdefault}{\mddefault}{\updefault}{\color[rgb]{0,0,0}$\ldots$}%
}}}}
\put(541,-3751){\makebox(0,0)[lb]{\smash{{\SetFigFont{12}{14.4}{\rmdefault}{\mddefault}{\updefault}{\color[rgb]{0,0,0}$B_1$}%
}}}}
\put(1261,-3661){\makebox(0,0)[lb]{\smash{{\SetFigFont{12}{14.4}{\rmdefault}{\mddefault}{\updefault}{\color[rgb]{0,0,0}$v_{126}$}%
}}}}
\put(2386,-3616){\makebox(0,0)[lb]{\smash{{\SetFigFont{12}{14.4}{\rmdefault}{\mddefault}{\updefault}{\color[rgb]{0,0,0}$v_{146}$}%
}}}}
\put(3736,-3211){\makebox(0,0)[lb]{\smash{{\SetFigFont{12}{14.4}{\rmdefault}{\mddefault}{\updefault}{\color[rgb]{0,0,0}$v_{236}$}%
}}}}
\put(3016,-2941){\makebox(0,0)[lb]{\smash{{\SetFigFont{12}{14.4}{\rmdefault}{\mddefault}{\updefault}{\color[rgb]{0,0,0}$F_3$}%
}}}}
\put(2566,-3031){\makebox(0,0)[lb]{\smash{{\SetFigFont{12}{14.4}{\rmdefault}{\mddefault}{\updefault}{\color[rgb]{0,0,0}$F_4$}%
}}}}
\put(2791,-3886){\makebox(0,0)[lb]{\smash{{\SetFigFont{12}{14.4}{\rmdefault}{\mddefault}{\updefault}{\color[rgb]{0,0,0}$B_2$}%
}}}}
\put(6661, 29){\makebox(0,0)[lb]{\smash{{\SetFigFont{12}{14.4}{\rmdefault}{\mddefault}{\updefault}{\color[rgb]{0,0,0}$v_{125}$}%
}}}}
\put(10666,-1231){\makebox(0,0)[lb]{\smash{{\SetFigFont{12}{14.4}{\rmdefault}{\mddefault}{\updefault}{\color[rgb]{0,0,0}$B_8$}%
}}}}
\put(10711,-961){\makebox(0,0)[lb]{\smash{{\SetFigFont{12}{14.4}{\rmdefault}{\mddefault}{\updefault}{\color[rgb]{0,0,0}$\bullet$}%
}}}}
\put(3061,-2401){\makebox(0,0)[lb]{\smash{{\SetFigFont{12}{14.4}{\rmdefault}{\mddefault}{\updefault}{\color[rgb]{0,0,0}$F_5$}%
}}}}
\put(3241,-2761){\makebox(0,0)[lb]{\smash{{\SetFigFont{12}{14.4}{\rmdefault}{\mddefault}{\updefault}{\color[rgb]{0,0,0}$v_{125}$}%
}}}}
\put(4321,-1726){\makebox(0,0)[lb]{\smash{{\SetFigFont{12}{14.4}{\rmdefault}{\mddefault}{\updefault}{\color[rgb]{0,0,0}Retraction sequence (a)}%
}}}}
\put(4186,-4066){\makebox(0,0)[lb]{\smash{{\SetFigFont{12}{14.4}{\rmdefault}{\mddefault}{\updefault}{\color[rgb]{0,0,0}Retraction sequence (b)}%
}}}}
\end{picture}%

%% file: torus1.pdf_t
\begin{picture}(0,0)%
\includegraphics{torus1.pdf}%
\end{picture}%
\setlength{\unitlength}{4144sp}%
\begingroup\makeatletter\ifx\SetFigFont\undefined%
\gdef\SetFigFont#1#2#3#4#5{%
  \reset@font\fontsize{#1}{#2pt}%
  \fontfamily{#3}\fontseries{#4}\fontshape{#5}%
  \selectfont}%
\fi\endgroup%
\begin{picture}(7500,3778)(166,-3995)
\put(1351,-3616){\makebox(0,0)[lb]{\smash{{\SetFigFont{12}{14.4}{\rmdefault}{\mddefault}{\updefault}{\color[rgb]{0,0,0}$(1, 1, -1)$}%
}}}}
\put(2746,-2041){\makebox(0,0)[lb]{\smash{{\SetFigFont{12}{14.4}{\rmdefault}{\mddefault}{\updefault}{\color[rgb]{0,0,0}$(-3, -1, 4)$}%
}}}}
\put(181,-1861){\makebox(0,0)[lb]{\smash{{\SetFigFont{12}{14.4}{\rmdefault}{\mddefault}{\updefault}{\color[rgb]{0,0,0}$(1, 0, 2)$}%
}}}}
\put(1306,-376){\makebox(0,0)[lb]{\smash{{\SetFigFont{12}{14.4}{\rmdefault}{\mddefault}{\updefault}{\color[rgb]{0,0,0}$(1, -2, 3)$}%
}}}}
\put(6256,-601){\makebox(0,0)[lb]{\smash{{\SetFigFont{12}{14.4}{\rmdefault}{\mddefault}{\updefault}{\color[rgb]{0,0,0}$(-3, -1, 4)$}%
}}}}
\put(4366,-2986){\makebox(0,0)[lb]{\smash{{\SetFigFont{12}{14.4}{\rmdefault}{\mddefault}{\updefault}{\color[rgb]{0,0,0}$(1, -2, 3)$}%
}}}}
\put(7426,-916){\makebox(0,0)[lb]{\smash{{\SetFigFont{12}{14.4}{\rmdefault}{\mddefault}{\updefault}{\color[rgb]{0,0,0}$(1, 1, -1)$}%
}}}}
\put(7156,-2896){\makebox(0,0)[lb]{\smash{{\SetFigFont{12}{14.4}{\rmdefault}{\mddefault}{\updefault}{\color[rgb]{0,0,0}$(3, -5, -1)$}%
}}}}
\put(4231,-916){\makebox(0,0)[lb]{\smash{{\SetFigFont{12}{14.4}{\rmdefault}{\mddefault}{\updefault}{\color[rgb]{0,0,0}$(-2, 1, -1)$}%
}}}}
\put(1441,-3931){\makebox(0,0)[lb]{\smash{{\SetFigFont{12}{14.4}{\rmdefault}{\mddefault}{\updefault}{\color[rgb]{0,0,0}(a)}%
}}}}
\put(4276,-2491){\makebox(0,0)[lb]{\smash{{\SetFigFont{12}{14.4}{\rmdefault}{\mddefault}{\updefault}{\color[rgb]{0,0,0}$A$}%
}}}}
\put(5311,-2986){\makebox(0,0)[lb]{\smash{{\SetFigFont{12}{14.4}{\rmdefault}{\mddefault}{\updefault}{\color[rgb]{0,0,0}$B$}%
}}}}
\put(6031,-3121){\makebox(0,0)[lb]{\smash{{\SetFigFont{12}{14.4}{\rmdefault}{\mddefault}{\updefault}{\color[rgb]{0,0,0}$(1, 0, 2)$}%
}}}}
\put(6706,-2941){\makebox(0,0)[lb]{\smash{{\SetFigFont{12}{14.4}{\rmdefault}{\mddefault}{\updefault}{\color[rgb]{0,0,0}$C$}%
}}}}
\put(7651,-2491){\makebox(0,0)[lb]{\smash{{\SetFigFont{12}{14.4}{\rmdefault}{\mddefault}{\updefault}{\color[rgb]{0,0,0}$D$}%
}}}}
\put(6706,-916){\makebox(0,0)[lb]{\smash{{\SetFigFont{12}{14.4}{\rmdefault}{\mddefault}{\updefault}{\color[rgb]{0,0,0}$E$}%
}}}}
\put(5311,-916){\makebox(0,0)[lb]{\smash{{\SetFigFont{12}{14.4}{\rmdefault}{\mddefault}{\updefault}{\color[rgb]{0,0,0}$F$}%
}}}}
\put(4276,-1411){\makebox(0,0)[lb]{\smash{{\SetFigFont{12}{14.4}{\rmdefault}{\mddefault}{\updefault}{\color[rgb]{0,0,0}$G$}%
}}}}
\put(7651,-1366){\makebox(0,0)[lb]{\smash{{\SetFigFont{12}{14.4}{\rmdefault}{\mddefault}{\updefault}{\color[rgb]{0,0,0}$H$}%
}}}}
\put(6121,-3796){\makebox(0,0)[lb]{\smash{{\SetFigFont{12}{14.4}{\rmdefault}{\mddefault}{\updefault}{\color[rgb]{0,0,0}(b)}%
}}}}
\end{picture}%

%% file: torus2.pdf_t
\begin{picture}(0,0)%
\includegraphics{torus2.pdf}%
\end{picture}%
\setlength{\unitlength}{4144sp}%
\begingroup\makeatletter\ifx\SetFigFont\undefined%
\gdef\SetFigFont#1#2#3#4#5{%
  \reset@font\fontsize{#1}{#2pt}%
  \fontfamily{#3}\fontseries{#4}\fontshape{#5}%
  \selectfont}%
\fi\endgroup%
\begin{picture}(8307,7812)(2191,-8581)
\put(4276,-2491){\makebox(0,0)[lb]{\smash{{\SetFigFont{12}{14.4}{\rmdefault}{\mddefault}{\updefault}{\color[rgb]{0,0,0}$A$}%
}}}}
\put(5311,-2986){\makebox(0,0)[lb]{\smash{{\SetFigFont{12}{14.4}{\rmdefault}{\mddefault}{\updefault}{\color[rgb]{0,0,0}$B$}%
}}}}
\put(6706,-2941){\makebox(0,0)[lb]{\smash{{\SetFigFont{12}{14.4}{\rmdefault}{\mddefault}{\updefault}{\color[rgb]{0,0,0}$C$}%
}}}}
\put(7651,-2491){\makebox(0,0)[lb]{\smash{{\SetFigFont{12}{14.4}{\rmdefault}{\mddefault}{\updefault}{\color[rgb]{0,0,0}$D$}%
}}}}
\put(6706,-916){\makebox(0,0)[lb]{\smash{{\SetFigFont{12}{14.4}{\rmdefault}{\mddefault}{\updefault}{\color[rgb]{0,0,0}$E$}%
}}}}
\put(5311,-916){\makebox(0,0)[lb]{\smash{{\SetFigFont{12}{14.4}{\rmdefault}{\mddefault}{\updefault}{\color[rgb]{0,0,0}$F$}%
}}}}
\put(4276,-1411){\makebox(0,0)[lb]{\smash{{\SetFigFont{12}{14.4}{\rmdefault}{\mddefault}{\updefault}{\color[rgb]{0,0,0}$G$}%
}}}}
\put(7651,-1366){\makebox(0,0)[lb]{\smash{{\SetFigFont{12}{14.4}{\rmdefault}{\mddefault}{\updefault}{\color[rgb]{0,0,0}$H$}%
}}}}
\put(5266,-5686){\makebox(0,0)[lb]{\smash{{\SetFigFont{12}{14.4}{\rmdefault}{\mddefault}{\updefault}{\color[rgb]{0,0,0}$B$}%
}}}}
\put(6616,-5686){\makebox(0,0)[lb]{\smash{{\SetFigFont{12}{14.4}{\rmdefault}{\mddefault}{\updefault}{\color[rgb]{0,0,0}$C$}%
}}}}
\put(7651,-5236){\makebox(0,0)[lb]{\smash{{\SetFigFont{12}{14.4}{\rmdefault}{\mddefault}{\updefault}{\color[rgb]{0,0,0}$D$}%
}}}}
\put(6706,-3616){\makebox(0,0)[lb]{\smash{{\SetFigFont{12}{14.4}{\rmdefault}{\mddefault}{\updefault}{\color[rgb]{0,0,0}$E$}%
}}}}
\put(5311,-3571){\makebox(0,0)[lb]{\smash{{\SetFigFont{12}{14.4}{\rmdefault}{\mddefault}{\updefault}{\color[rgb]{0,0,0}$F$}%
}}}}
\put(4276,-4066){\makebox(0,0)[lb]{\smash{{\SetFigFont{12}{14.4}{\rmdefault}{\mddefault}{\updefault}{\color[rgb]{0,0,0}$G$}%
}}}}
\put(7606,-4021){\makebox(0,0)[lb]{\smash{{\SetFigFont{12}{14.4}{\rmdefault}{\mddefault}{\updefault}{\color[rgb]{0,0,0}$H$}%
}}}}
\put(5266,-8161){\makebox(0,0)[lb]{\smash{{\SetFigFont{12}{14.4}{\rmdefault}{\mddefault}{\updefault}{\color[rgb]{0,0,0}$B$}%
}}}}
\put(6616,-8206){\makebox(0,0)[lb]{\smash{{\SetFigFont{12}{14.4}{\rmdefault}{\mddefault}{\updefault}{\color[rgb]{0,0,0}$C$}%
}}}}
\put(7561,-7711){\makebox(0,0)[lb]{\smash{{\SetFigFont{12}{14.4}{\rmdefault}{\mddefault}{\updefault}{\color[rgb]{0,0,0}$D$}%
}}}}
\put(7606,-6496){\makebox(0,0)[lb]{\smash{{\SetFigFont{12}{14.4}{\rmdefault}{\mddefault}{\updefault}{\color[rgb]{0,0,0}$H$}%
}}}}
\put(6661,-6091){\makebox(0,0)[lb]{\smash{{\SetFigFont{12}{14.4}{\rmdefault}{\mddefault}{\updefault}{\color[rgb]{0,0,0}$E$}%
}}}}
\put(5311,-6091){\makebox(0,0)[lb]{\smash{{\SetFigFont{12}{14.4}{\rmdefault}{\mddefault}{\updefault}{\color[rgb]{0,0,0}$F$}%
}}}}
\put(3961,-6046){\makebox(0,0)[lb]{\smash{{\SetFigFont{12}{14.4}{\rmdefault}{\mddefault}{\updefault}{\color[rgb]{0,0,0}$E$}%
}}}}
\put(4906,-6316){\makebox(0,0)[lb]{\smash{{\SetFigFont{12}{14.4}{\rmdefault}{\mddefault}{\updefault}{\color[rgb]{0,0,0}$H$}%
}}}}
\put(4906,-7261){\makebox(0,0)[lb]{\smash{{\SetFigFont{12}{14.4}{\rmdefault}{\mddefault}{\updefault}{\color[rgb]{0,0,0}$D$}%
}}}}
\put(4051,-7711){\makebox(0,0)[lb]{\smash{{\SetFigFont{12}{14.4}{\rmdefault}{\mddefault}{\updefault}{\color[rgb]{0,0,0}$C$}%
}}}}
\put(2701,-7666){\makebox(0,0)[lb]{\smash{{\SetFigFont{12}{14.4}{\rmdefault}{\mddefault}{\updefault}{\color[rgb]{0,0,0}$B$}%
}}}}
\put(2206,-6406){\makebox(0,0)[lb]{\smash{{\SetFigFont{12}{14.4}{\rmdefault}{\mddefault}{\updefault}{\color[rgb]{0,0,0}$G$}%
}}}}
\put(7831,-6181){\makebox(0,0)[lb]{\smash{{\SetFigFont{12}{14.4}{\rmdefault}{\mddefault}{\updefault}{\color[rgb]{0,0,0}$G$}%
}}}}
\put(8596,-5596){\makebox(0,0)[lb]{\smash{{\SetFigFont{12}{14.4}{\rmdefault}{\mddefault}{\updefault}{\color[rgb]{0,0,0}$F$}%
}}}}
\put(9631,-5596){\makebox(0,0)[lb]{\smash{{\SetFigFont{12}{14.4}{\rmdefault}{\mddefault}{\updefault}{\color[rgb]{0,0,0}$E$}%
}}}}
\put(10441,-5911){\makebox(0,0)[lb]{\smash{{\SetFigFont{12}{14.4}{\rmdefault}{\mddefault}{\updefault}{\color[rgb]{0,0,0}$H$}%
}}}}
\put(10441,-6991){\makebox(0,0)[lb]{\smash{{\SetFigFont{12}{14.4}{\rmdefault}{\mddefault}{\updefault}{\color[rgb]{0,0,0}$D$}%
}}}}
\put(9361,-7261){\makebox(0,0)[lb]{\smash{{\SetFigFont{12}{14.4}{\rmdefault}{\mddefault}{\updefault}{\color[rgb]{0,0,0}$C$}%
}}}}
\put(5581,-8566){\makebox(0,0)[lb]{\smash{{\SetFigFont{12}{14.4}{\rmdefault}{\mddefault}{\updefault}{\color[rgb]{0,0,0}So on ...}%
}}}}
\end{picture}%